\documentclass[a4paper,reqno]{amsart}

\usepackage{eufrak}
\usepackage{graphicx}
\usepackage{amssymb}
\usepackage{amscd}
\usepackage{amsthm}
\usepackage{amsmath}
\usepackage{enumerate}
\usepackage{mathrsfs}
\usepackage{graphics}
\usepackage[all]{xy}
\usepackage{tikz}
\usepackage{extarrows}
\usepackage{tikz-cd}
\usepackage[normalem]{ulem}
\usetikzlibrary{calc}
\usetikzlibrary{matrix,arrows,decorations.pathmorphing}
\usetikzlibrary{arrows,decorations.markings, matrix}
\numberwithin{equation}{section}

\usepackage[colorlinks=true,pagebackref,hyperindex]{hyperref}
\newcommand\myshade{85}
\colorlet{mylinkcolor}{violet}
\colorlet{mycitecolor}{red}
\colorlet{myurlcolor}{cyan}

\hypersetup{
  linkcolor  = mylinkcolor!\myshade!black,
  citecolor  = mycitecolor!\myshade!black,
  urlcolor   = myurlcolor!\myshade!black,
  colorlinks = true,
}

\newtheorem{Thm}{Theorem}[section]
\newtheorem*{Thm*}{Theorem}
\newtheorem{Lem}[Thm]{Lemma} \newtheorem{Def}[Thm]{Definition}
\newtheorem{Cor}[Thm]{Corollary}
\newtheorem{Prop}[Thm]{Proposition}
\newtheorem{Prop-Def}[Thm]{Proposition-Definition}
\newtheorem{Ex}[Thm]{Example}
\newtheorem{Rem}[Thm]{Remark}

\theoremstyle{definition}

\newcommand{\Mod}{\mathrm{Mod}}

\newcommand{\proj}{\mathrm{proj}}
\newcommand{\Hom}{\mathrm{Hom}}
\newcommand{\cHom}{{\mathcal H}om}

\newcommand{\RHom}{\mathbf{R}\mathrm{Hom}}

\newcommand{\w}{\widetilde}
\newcommand{\ov}{\overline}

\newcommand{\lra}{\longrightarrow}

\newcommand{\ra}{\rightarrow}
\newcommand{\sdp}{\times\kern-.2em\vrule height1.1ex depth-.05ex}
\newcommand{\epi}{\lra \kern-.8em\ra}

\newcommand{\T}{\mathcal T}

\newcommand{\cal}{\mathcal}
\newcommand{\Z}{{\mathbb Z}}

\newcommand{\D}{\mathscr{D}}
\newcommand{\K}{\mathscr{K}}
\newcommand{\C}{\mathscr{C}}
\renewcommand{\L}{\mathbf{L}}

\newcommand{\A}{\mathcal{A}}
\newcommand{\B}{\mathcal{B}}
\newcommand{\h}{{\mathrm H}}

\renewcommand{\mod}{\mathrm{mod}}
\newcommand{\per}{\mathsf {per}}
\newcommand{\fd}{\mathrm{fd}}
\newcommand{\fg}{\mathrm{fg}}
\newcommand{\sg}{\mathrm{sg}}
\newcommand{\dg}{\mathrm{dg}}
\newcommand{\agk}{\mathrm{agk}}
\newcommand{\tr}{\mathrm{tr}}
\newcommand{\thick}{\mathsf{thick}}
\newcommand{\Loc}{\mathsf{Loc}}
\newcommand{\ke}{\mathsf{Ker}}
\newcommand{\im}{\mathsf{Im}}
\newcommand{\cok}{\mathsf{cok}}
\newcommand{\LT}{\mathbf{L}T}

\renewcommand{\bf}{\mathbf}
\newcommand{\oop}{\rm op}

\newcommand{\bop}{\bigoplus}

\newcommand{\xra}{\xrightarrow}
\newcommand{\ten}{\otimes}
\newcommand{\lten}{\otimes^{\bf L}}
\renewcommand{\to}{\rightarrow}
\DeclareMathOperator{\add}{\mathsf {add}}

\newcommand{\gl}{\mathop{{\rm gl.dim\,}}\nolimits}

\newcommand{\RshHom}{\mathbf{R}\strut\kern-.2em\mathscr{H}\strut\kern-.3em\operatorname{om}\nolimits}
\newcommand{\shHom}{\mathscr{H}\strut\kern-.3em\operatorname{om}\nolimits}
\newcommand{\shEnd}{\mathscr{E}\strut\kern-.3em\operatorname{nd}\nolimits}
\newcommand{\cEnd}{{\mathcal{E}nd}}

\newcommand{\za}{\alpha}

\newcommand{\ie}{\emph{i.e.~}}

\definecolor{dark-green}{RGB}{14,150,2}
\definecolor{red}{RGB}{250,0,0}

\newcommand{\jin}{\color{purple}}

\begin{document}

\title{A localisation theorem for singularity categories of proper dg algebras}
\author{Haibo Jin, Dong Yang and Guodong Zhou}

\dedicatory{Dedicated to Bernhard Keller on the occasion of his 60th birthday}

\address{Haibo Jin, Department of Mathematics, Universit\"at zu K\"oln, Weyertal 86-90, 50931 K\"oln, Germany}
\email{hjin@math.uni-koeln.de}

\address{Dong Yang, Department of Mathematics, Nanjing University, 22 Hankou Road, Nanjing 210093, China}
\email{yangdong@nju.edu.cn}

\address{Guodong Zhou, School of Mathematical Sciences,  Key Laboratory of Mathematics and Engineering Applications(Ministry of Education), Shanghai Key Laboratory of PMMP,  East China Normal University, Shanghai 200241, China}
\email{gdzhou@math.ecnu.edu.cn}

  \renewcommand{\thefootnote}{\alph{footnote}}
  \setcounter{footnote}{-1} 
  \setcounter{footnote}{-1} 

 \begin{abstract}
Given a recollement of three proper dg algebras over a noetherian commutative ring, e.g. three algebras which are finitely generated as modules over the base ring, which extends one step downwards, it is shown that there is a short exact sequence of their singularity categories.  This allows us to recover and generalise some known results on singularity categories of finite-dimensional algebras.
 \end{abstract}
 \maketitle
\setcounter{tocdepth}{1}
 \tableofcontents

 \section{Introduction}

For a right noetherian ring, the \emph{singularity category} (\cite{Buchweitz,Orlov04}) of $A$ is defined as the triangle quotient
\[
\D_{\sg}(A):=\D^b(\mod A)/\K^b(\proj A),
\]
where $\D^b(\mod A)$ and $\K^b(\proj A)$ are the bounded derived category of finitely generated right $A$-modules and the bounded homotopy category of finitely generated projective right $A$-modules, respectively. It measures the degree to which $A$ is `singular' \cite{Orlov04}, and captures the stable homological features of $A$ \cite{Buchweitz}. In this paper we show the following localisation theorem.

\begin{Thm}
\label{thm:main-result-1}
Let $k$ be a commutative noetherian ring and $A,B$ and $C$ be $k$-algebras which are finitely generated as $k$-modules. Assume that there is a recollement of unbounded derived categories in the sense of \cite{BBD}
\begin{equation}\label{equ:recolintro}
 \xymatrixcolsep{4pc}\xymatrix{\D(B) \ar[r] &\D(A) \ar@/_1pc/[l] \ar@/^1pc/[l] \ar[r]  &\D(C) \ar@/^-1pc/[l] \ar@/^1pc/[l]}
\end{equation}
\vspace{1pt}

\noindent which extends one step downwards in the obvious sense. Then the second row induces a short exact sequence of triangulated categories
 \[\label{equ:sesDsg}
 \xymatrix{
 \D_{\sg}(B) \ar[r] &  \D_{\sg}(A) \ar[r] & \D_{\sg}(C).}
 \]
\end{Thm}
  Here a sequence of triangulated categories $\mathcal{T'}\xleftarrow[]{F}\mathcal{T}\xleftarrow[]{G}\mathcal{T''}$ is called a \emph{short exact sequence (up to direct summands)} if G is fully faithful, $F\circ G=0$ and the induced functor $\overline{F}\colon \mathcal{T}/\mathcal{T''}\ra\mathcal{T'}$ is an equivalence (up to direct summands). A consequence of Theorem~\ref{thm:main-result-1} is

\begin{Cor}
\label{cor:triangular-matrix-algebra}
Let $B$ and $C$ be $k$-algebras which are finitely generated as $k$-modules and $M$ be a $C$-$B$-bimodule which is finitely generated on both sides. Consider the matrix algebra $A=\left( \begin{array}{cc} B & 0\\ M & C\end{array}\right)$. Then there is a short exact sequence of triangulated categories
 \[\label{equ:sesDsg}
 \xymatrix{
 \D_{\sg}(B) \ar[r] &  \D_{\sg}(A) \ar[r] & \D_{\sg}(C).}
 \]
\end{Cor}

Theorem~\ref{thm:main-result-1} and Corollary~\ref{cor:triangular-matrix-algebra} allow us to recover and generalise some known results on singularity categories \cite{Chen09,Chen10,Chen14,LiuLu15,Qin20}. Also as an application of Theorem~\ref{thm:main-result-1}, the singularity categories of  finite-dimensional graded monomial algebras are studied in \cite{Jin-Schroll}.


We state and prove Theorem~\ref{thm:main-result-1} in the more general setting of proper dg algebras (Theorem~\ref{Thm:recoltoDsg}), and present a similar result for the AGK category of a homologically smooth dg algebra (Theorem~\ref{Thm:sescluster}). 
We also establish a further generalisation.

\begin{Thm}[Theorem~\ref{thm:s.e.s.-of-quotients-in-the-first-row}]
\label{Thm:First Main Result}
Assume the following data are given:
   \begin{itemize}
     \item $\mathcal{T}$, $\mathcal{T}'$, $\mathcal{T}''$ are compactly generated triangulated $k$-categories together with a recollement of $\mathcal{T}$ in terms of $\mathcal{T}'$ and $\mathcal{T}''$,
where $\mathcal{T}$ is algebraic;

    \item $\cal S$, $\cal S'$ and $\cal S''$ are triangulated subcategories of $\mathcal{T}$, $\mathcal{T}'$ and $\mathcal{T}''$,  which  contain $\mathcal{T}^c$, $\mathcal{T}'^c$ and $\mathcal{T}''^c$, respectively;
    \item the first row of the recollement restricts to a short exact sequence up to direct summands $\cal S'\leftarrow \cal S\leftarrow\cal S''$.
    \end{itemize}
Then the first row of the recollement induces a short exact sequence of triangulated categories up to  direct summands:
\[
\xymatrix{
\mathcal{S}'/\mathcal{T}'^c &  \mathcal{S}/\mathcal{T}^c \ar[l] & \mathcal{S}''/\mathcal{T}''^c \ar[l].
}
\]
\end{Thm}

In Section~\ref{Sect: geometrical application}, we present an application of Theorem~\ref{Thm:First Main Result} to algebraic geometry by providing a new proof of the following result. 

\begin{Thm}[{\cite[Theorem 1.3]{Chen10}}] Let $X$ be a scheme satisfying Orlov's condition (ELF).
Let  $U\stackrel{i}{\hookrightarrow} X$ an open subscheme and write  $Z=X\backslash
U$.
      Then there exists a short exact sequence  of  triangulated categories: 
\[ \xymatrix{\D_{\sg}(U) & \D_{\sg}(X) \ar[l]|{i^*} & \D_{\sg, Z} (X) \ar[l]|{j_!}. }
\]

  \end{Thm}

\medskip

\subsection*{Notation}
Throughout, let $k$ be a commutative ring.
Let $\mathcal{T}$ be a triangulated $k$-category. For a set $\mathcal{S}$ of objects of $\mathcal{T}$, denote 
by $\thick(\mathcal{S})=\thick_{\mathcal{T}}(\mathcal{S})$ the smallest triangulated subcategory of $\mathcal{T}$ which contains $\mathcal{S}$ and which is closed under taking direct summands. 
For a triangle functor $F\colon\mathcal{T}\to\mathcal{T}'$, denote
\[
\ke(F):=\{X\in\mathcal{T}\mid F(X)\cong 0\}.
\]
Assume that $\mathcal{T}$ has infinite direct sums.
An object $X$ of $\T$ is \emph{compact}, if $\Hom_{\T}(X,?)$ commutes with infinite direct sums. We denote by $\mathcal{T}^c$ the subcategory of $\mathcal{T}$ consisting of compact objects. For  a set $\mathcal{S}$ of objects of $\mathcal{T}$, denote by $\Loc(\mathcal{S})=\Loc_{\mathcal{T}}(\mathcal{S})$ the smallest triangulated subcategory of $\mathcal{T}$ which contains $\mathcal{S}$ and which is closed under taking direct sums.

 \bigskip

\noindent\textbf{Acknowledgement}: The first author was partially supported by  the Deutsche Forschungsgemeinschaft (DFG)  through the project SFB/TRR 191 Symplectic Structures
in Geometry, Algebra and Dynamics (Projektnummer 281071066--TRR 191). The first and the third author were supported by the National Natural Science Foundation of China (No.  12071137), by  Key Laboratory of MEA(Ministry of Education) and by  Shanghai Key Laboratory of PMMP  (No.   22DZ2229014). The second author was supported by the National Natural Science Foundation of China (No. 12031007) and the DFG program SPP 1388 (YA297/1-1 and KO1281/9-1).

\bigskip

 \section{Derived categories of dg algebras and recollements}

In this section we recall and establish some basic results on derived categories of dg algebras and on recollements.

\subsection{Derived categories of dg algebras}
\label{ss:derived-category-of-dg-algebras}\

We follow \cite{Keller94,Keller06}.
Let $A$ be a dg $k$-algebra. By dg $A$-modules we will mean right dg $A$-modules, and we will identify left dg $A$-modules with right dg $A^{\oop}$-modules, where $A^{\oop}$ is the opposite dg algebra of $A$. A $k$-algebra $A$ can be viewed as a dg $k$-algebra concentrated in degree $0$. In this case, a dg $A$-module is exactly a complex of $A$-modules. 

For two dg $A$-modules $M$ and $N$, let $\cHom_A(M,N)$ be the complex of $k$-modules with $\cHom_A^p(M,N)$ consisting of the homogeneous $A$-linear maps $M\to N$ of degree $p$ and with differential taking $f\in\cHom_A^p(M,N)$ to $d_N\circ f-(-1)^p f\circ d_M$. The \emph{dg category of dg $A$-modules} $\C_{\dg}(A)$ has dg $A$-modules as its objects and for two dg $A$-modules
\[
\Hom_{\C_{\dg}(A)}(M,N):=\cHom_A(M,N).
\]
The \emph{homotopy category of dg $A$-modules} $\K(A)$ has dg $A$-modules as its objects and for two dg $A$-modules
\[
\Hom_{\K(A)}(M,N):=\h^0\cHom_A(M,N).
\]
A dg $A$-module $M$ is said to be \emph{acyclic} if $\h^p(M)$ vanishes for any $p\in\mathbb{Z}$.
The \emph{derived category of dg $A$-modules} $\D(A)$ is the triangle quotient of $\D(A)$ by the full subcategory of acyclic dg $A$-modules. Put $\per(A):=\thick_{\D(A)}(A_A)$, called the \emph{perfect derived category} of $A$. Then $\D(A)$ has infinite direct sums and is compactly generated by $A_A$, and $\per(A)=\D(A)^c$ (see \cite[Section 5]{Keller94}). If $A$ is a $k$-algebra, considered as a dg $k$-algebra concentrated in degree $0$, then $\D(A)=\D(\Mod A)$ is exactly the unbounded derived category of the category $\Mod A$ of all $A$-modules, and there is a canonical triangle equivalence from the bounded homotopy category $\K^b(\proj A)$ of finitely generated projective $A$-modules to $\per(A)$.

A dg $A$-module $P$ is said to be \emph{$\K$-projective} if $\cHom_A(P,?)$ takes an acyclic dg $A$-module to an acyclic complex of $k$-modules. For example, the free dg $A$-module $A_A$ of rank $1$ is $\K$-projective. It follows from \cite[Theorem 3.1]{Keller94} that for any dg $A$-module $M$ there is always a quasi-isomorphism $P_M\to M$ of dg $A$-modules with $P_M$ being $\K$-projective. Such a quasi-isomorphism is called a \emph{$\K$-projective resolution} of $M$. If $P$ is a $\K$-projective dg $A$-module and $N$ is any dg $A$-module, then
\begin{align}
\label{eq:morphism-in-derived-category}
\Hom_{\D(A)}(P,N)\cong\Hom_{\K(A)}(P,N)=\h^0\cHom_A(P,N).
\end{align}
In particular, $\Hom_{\D(A)}(A,A[p])=\h^p(A)$ for $p\in\mathbb{Z}$.

\smallskip

\subsubsection{Derived functors}\

Let $A$ and $B$ be two dg $k$-algebras and $X$ be a dg $B$-$A$-bimodule. Then there is an adjoint pair of triangle functors
\[
\xymatrix@C=6pc{
\D(A)  \ar@<-1ex>[r]_{\RHom_A(X,?)}& \D(B)\ar@<-1ex>[l]_{?\lten_B X} .
}
\]
Assume that $X_A$ is $\K$-projective and that $X_A\in\per(A)$. Then by \cite[Lemma 6.2]{Keller94}, there is an isomorphism $\RHom_A(X,?)\cong ?\lten_A X^{\tr}$ of triangle functors, where $X^{\tr}=X^{\tr_A}:=\cHom_A(X,A)$. The following lemma is immediate.

\begin{Lem}
\label{lem:generating-adjoint-triple}
Assume that $X_A$ is $\K$-projective and that $X_A\in\per(A)$. Then there is an adjoint triple of triangle functors
\[
\xymatrixcolsep{5pc}\xymatrix{
\D(A)  \ar[r]|{?\lten_A X^{\tr}} &\D(B). \ar@/_1pc/[l]_{?\lten_B X} \ar@/^1pc/[l]^{\RHom_A(X^{\tr},?)}
  } \]
We call this triple the adjoint triple induced by $X$.
\end{Lem}

The following lemma is a generalisation of \cite[Lemma 2.8 (i)$\Rightarrow$(ii)]{AKLY}. The proof is similar, so we omit it.
\begin{Lem}
\label{lem:left-adjoint-of-tensor}
Assume that ${}_BX$ is $\K$-projective and that ${}_B X\in\per(B^{\oop})$. Then $?\lten_B X\colon \D(B)\to\D(A)$ has a left adjoint $?\lten_A X^{\mathrm{tr}_{B^{\oop}}}$.
\end{Lem}

Let $M$ be a dg $A$-module and put $\cEnd_A(M):=\cHom_A(M,M)$. Then $M$ naturally becomes a dg $\cEnd_A(M)$-$A$-bimodule. The following lemma is well known.

\begin{Lem}
\label{lem:compact-object-induces-embedding}
The triangle functor $?\lten_{\cEnd_A(M)} M\colon \D(\cEnd_A(M))\to\D(A)$ is fully faithful if $M_A\in\per(A)$ and $M_A$ is $\K$-projective.
\end{Lem}

\smallskip

\subsubsection{The finitely generated derived category}\

Assume that $k$ is noetherian. Let $\D_{\fg}(A)$ denote the full subcategory of $\D(A)$ consisting of dg $A$-modules $M$ such that $\h^*(M)$ is finitely generated over $k$. This is a triangulated subcategory of $\D(A)$. If $A$ is a $k$-algebra (viewed as a dg $k$-algebra concentrated in degree $0$) and is finitely generated as a $k$-module, then there is a natural triangle equivalence $\D^b(\mod A)\to\D_{\fg}(A)$. When $k$ is a field, we also denote this category by $\D_{\fd}(A)$. The following description of $\D_{\fg}(A)$ is obtained by d\'evissage, similar to \cite[Lemmas 2.4(a)]{AKLY}.

\begin{Lem}
\label{lem:description-of-Dfd}
Let $A$ be a dg $k$-algebra. Then for $X\in\D(A)$ the following conditions are equivalent:
\begin{itemize}
\item[(i)] $X\in \D_{\fg}(A)$,
\item[(ii)] $\h^*\RHom_A(P,X)$ is finitely generated over $k$ for any $P\in \per(A)$, \ie $\bigoplus_{p\in\mathbb{Z}}\Hom_{\D(A)}(P,X[p])$ is finitely generated over $k$. Here we consider $P$ as a dg $k$-$A$-bimodule.
\item[(iii)] $\h^*\RHom_A(T,X)$ is finitely generated over $k$ for a/any classical generator $T$ of $\per(A)$ (\ie $\thick(T)=\per(A)$).
\end{itemize}
\end{Lem}

Thanks to this description we have

\begin{Lem}
\label{FkeepcompactGkeepDbmod}
Let $A, B$ be dg $k$-algebras. Let $(F,G)$ be an adjoint pair of triangle functors between their derived categories:
~$\xymatrix{\D(A)\ar@<.7ex>[r]^F & \D(B) \ar@<.7ex>[l]^G}$.
 If $F(\per(A))\subset \per(B)$, then  $G(\D_{\fg}(B))\subset \D_{\fg}(A)$. If in addition $F(A)$ is a classical generator of $\per(B)$, then $G$ detects $\D_\fg$, \ie for $N\in\D(B)$, if $G(N)\in\D_{\fg}(A)$, then $N\in\D_{\fg}(B)$.
\end{Lem}

\begin{proof}
For $N\in \D(B)$, we have
 \[\h^*(G(N))=\bop_{p\in\Z}\Hom_{\D(A)}(A, G(N)[p])\cong\bop_{p\in\Z}\Hom_{\D(B)}(F(A), N[p]).\]
If $F(A)\in \per(B)$ and $N\in\D_{\fd}(B)$, then by Lemma~\ref{lem:description-of-Dfd} we see that this $k$-module is finitely generated, showing that $G(N)\in \D_{\fg}(A)$. If $F(A)$ is a classical generator of $\per(B)$ and $G(N)\in\D_{\fg}(A)$, \ie $\thick(F(A))=\per(B)$, then by Lemma~\ref{lem:description-of-Dfd} that $N\in\D_{\fg}(B)$.
\end{proof}

\smallskip

\subsubsection{The Nakayama functor}\

Assume that $k$ is a field and let $A$ be a dg $k$-algebra. Consider the \emph{Nakayama functor}
\[
\nu=\nu_A=D\cHom_A(?,A)\colon \C_{\dg}(A)\to \C_{\dg}(A)
\]
and its left derived functor
\[
\mathbf{L}\nu=\mathbf{L}\nu_A\colon\D(A)\to \D(A).
\]
It is clear that $\mathbf{L}\nu(A)=DA$. So $\mathbf{L}\nu$ restricts to a triangle functor
\[
\mathbf{L}\nu\colon\per(A)\to \thick(DA),
\]
which is a triangle equivalence if $\h^pA$ is finite-dimensional for any $p\in\mathbb{Z}$. Moreover, we have the Auslander--Reiten formula, \ie a bifunctorial isomorphism
\[
D\Hom_{\D(A)}(M,N)\cong \Hom_{\D(A)}(N,\mathbf{L}\nu(M))
\]
for $M\in\per(A)$ and $N\in\D(A)$ (See \cite[Section 10]{Keller94} or \cite[Section 2.3]{KalckYang16}).

\begin{Lem}
\label{lem:Nakayama-functor}
Let $A$ and $B$ be dg $k$-algebras and assume that there is an adjoint triple of triangle functors
\[
\xymatrixcolsep{4pc}\xymatrix{
\D(B) \ar[r]|{G} &\D(A).\ar@/_1pc/[l]_{F} \ar@/^1pc/[l]^{H}
  }
\]
\begin{itemize}
\item[(a)] $F$ restricts to a triangle functor $F|_{\per(A)}\colon\per(A)\to\per(B)$, and there is an isomorphism of triangle functors
\[
(H\circ\mathbf{L}\nu_A)|_{\per(A)}\cong(\mathbf{L}\nu_B\circ F)|_{\per(A)}\colon\per(A)\to\D(B).
\]
\item[(b)] $H$ restricts to a triangle functor $H|_{\thick(DA)}\colon \thick(DA)\to\thick(DB)$, and the following diagram is commutative up to isomorphism of triangle functors
\[
\xymatrix{
\per(B)\ar[d]^{\mathbf{L}\nu_B}  & \per(A)\ar[l]_F\ar[d]^{\mathbf{L}\nu_A}\\
\thick(DB) & \thick(DA)\ar[l]_H
}
\]
\item[(c)] If $F$ is a triangle equivalence, then $F|_{\thick(DA)}\colon\thick(DA)\to\thick(DB)$ is a triangle equivalence.
\end{itemize}
\end{Lem}
\begin{proof}
(a) First, by Lemma \ref{adjoint} below, $F$ sends compact objects to compact objects, namely, $F$ restricts to a triangle functor $F|_{\per(A)}\colon \per(A)\to\per(B)$.
Now there are isomorphisms for $M\in\per(A)$ and $N\in\D(B)$,
\begin{align*}
\Hom_{\D(B)}(N,H\circ\mathbf{L}\nu_A(M))&\cong \Hom_{\D(A)}(G(N),\mathbf{L}\nu_A(M))\\
&\cong D\Hom_{\D(A)}(M,G(N))\\
&\cong D\Hom_{\D(B)}(F(M),N)\\
&\cong \Hom_{\D(B)}(N,\mathbf{L}\nu_B\circ F(M)),
\end{align*}
which are functorial in $M$ and $N$ and commute with the shift functors.

(b) In view of (a) it suffices to show that $H(DA)\in\thick(DB)$, which is again by (a): $H(DA)=H\circ\mathbf{L}\nu_A(A)\cong\mathbf{L}\nu_B\circ F(A)\in\thick(DB)$.

(c) If $F$ is an equivalence, then $F\cong H$ and thus $F(\thick(DA))\subset\thick(DB)$ by (b). Of course, the restriction $F|_{\thick(DA)}\colon\thick(DA)\to\thick(DB)$ is fully faithful. Moreover, there exists $M\in\per(A)$ such that $B_B\cong F(M)$. So
\[
DB=\mathbf{L}\nu_B(B)\cong\mathbf{L}\nu_B\circ F(M)\cong F\circ\mathbf{L}\nu_A(M)\in F(\thick(DA)),
\]
which show that $F|_{\thick(DA)}\colon\thick(DA)\to\thick(DB)$ is an equivalence by \cite[Lemma 4.2]{Keller94}.
\end{proof}

\medskip

\subsection{Recollements}\

In this subsection, we recall and establish some basic results on recollements.
Let $\mathcal{T}$, $\mathcal{T'}$ and $\mathcal{T''}$ be triangulated $k$-categories.

\begin{Def}[{\cite{BBD,Parshall89,Koenig91}}]~\label{recollement}
\label{defn:recollement}
 A \emph{recollement} of $\mathcal{T}$ in terms of $\mathcal{T'}$ and $\mathcal{T''}$ is a diagram of six triangle functors
\begin{align}
\label{recol:defn}
\xymatrixcolsep{4pc}\xymatrix{\mathcal{T'} \ar[r]|{i_*} &\mathcal{T} \ar@/_1pc/[l]_{i^*} \ar@/^1pc/[l]^{i^!} \ar[r]|{j^*}  &\mathcal{T''}, \ar@/^-1pc/[l]_{j_!} \ar@/^1pc/[l]^{j_{*}}
  }
\end{align}
  which satisfies

  \begin{itemize}
    \item[(R1)] $(i^*,~i_*,~i^!)$ and $(j_!,~j^*,~j_*)$ are adjoint triples$;$
    \item[(R2)] $i_*,~j_*,~j_!$ are fully faithful$;$
    \item[(R3)] $j^*\circ i_*=0;$
    \item[(R4)] for any object $X$ in $\mathcal{T}$, there are two triangles$:$
    \[ \xymatrix@R=0.5pc{i_*i^!X \ar[r] &X \ar[r] & j_*j^*X \ar[r]& i_*i^!X[1]\\
    j_!j^*X \ar[r] &X \ar[r] & i_*i^*X \ar[r]& j_!j^*X[1]} \]
    where the left four morphisms are given by the units and conuits.
  \end{itemize}
The upper two rows
\begin{align*}
\xymatrixcolsep{4pc}\xymatrix{\mathcal{T'} \ar[r]|{i_*} &\mathcal{T} \ar@/_1pc/[l]_{i^*}  \ar[r]|{j^*}  &\mathcal{T''}, \ar@/^-1pc/[l]_{j_!}
  }
\end{align*}
is called a \emph{left recollemet} of $\mathcal{T}$ in terms of $\mathcal{T'}$ and $\mathcal{T''}$ if the four functors $i^*$, $i_*$, $j_!$ and $j^*$ satisfy the conditions in (R1)--(R4) involving them.  A \emph{right recollement} is defined similarly via the lower two rows.
\end{Def}

Notice that a right recollement of $\mathcal{T}$ in terms of $\mathcal{T}'$ and $\mathcal{T}''$ is exactly a left recollement of $\mathcal{T}$ in terms of $\mathcal{T}''$ and $\mathcal{T}'$.

\begin{Lem}
\label{lem:restricting-half-recollement}
Let
\begin{align*}
\xymatrixcolsep{4pc}\xymatrix{\mathcal{T'} \ar[r]|{i_*} &\mathcal{T} \ar@/_1pc/[l]_{i^*}  \ar[r]|{j^*}  &\mathcal{T''} \ar@/^-1pc/[l]_{j_!}
  }
\end{align*}
be a left recollement of $\mathcal{T}$ in terms of $\mathcal{T}'$ and $\mathcal{T}''$. Assume that $\mathcal{S}$, $\mathcal{S}'$ and $\mathcal{S}''$ are subcategories of $\mathcal{T}$, $\mathcal{T}'$ and $\mathcal{T}''$, respectively, such that $i_*(\mathcal{S}')\subset\mathcal{S}$, $j^*(\mathcal{S})\subset \mathcal{S}''$, $j_!(\mathcal{S}'')\subset \mathcal{S}$ and $i^*(\mathcal{S})\subset \mathcal{S}'$. Then the above left recollement restricts to a left recollement of $\mathcal{S}$ in terms of $\mathcal{S}'$ and $\mathcal{S}''$.
\end{Lem}

The following result is well-known, see \cite[1.4.4, 1.4.5, 1.4.8]{BBD} and \cite[Section 2]{Miyachi91}.

\begin{Prop}
\label{prop:recollement-vs-ses}
\begin{itemize}
\item[(a)] The two rows of a left recollement are short exact sequences of triangulated categories.

Conversely, assume that there is a short exact sequence of triangulated categories (possibly up to direct summands)
\[
\xymatrixcolsep{4pc}\xymatrix{\mathcal{T'} \ar[r]|{i_*} &\mathcal{T}   \ar[r]|{j^*}  &\mathcal{T''}.
  }
\]
Then $i_*$ has a left adjoint (respectively, right adjoint) if and only if $j^*$ has a left adjoint (respectively, right adjoint). In this case, $i_*$ and $j^*$ together with their left adjoints (respectively, right adjoints) form a left recollement (respectively, right recollement) of $\mathcal{T}$ in terms of $\mathcal{T}'$ and $\mathcal{T}''$.
\item[(b)]
Assume that there is a diagram
\begin{align*}
\xymatrixcolsep{4pc}\xymatrix{\mathcal{T'} \ar[r]|{i_*} &\mathcal{T} \ar@/_1pc/[l]_{i^*} \ar@/^1pc/[l]^{i^!} \ar[r]|{j^*}  &\mathcal{T''}, \ar@/^-1pc/[l]_{j_!} \ar@/^1pc/[l]^{j_{*}}
  }
\end{align*}
satisfying the condition (R1). If it is a recollement, then all the three rows are short exact sequences of triangulated categories. Conversely, if any one of the three rows is a short exact sequence of triangulated categories, then the diagram is a recollement.
\end{itemize}
\end{Prop}

We say that a recollement \eqref{recol:defn} \emph{extends one step downwards}  if both $i^!$ and $j_*$ have right adjoints. In this case, the diagram extends to four rows and the lower three rows, by Proposition~\ref{prop:recollement-vs-ses}, also form a recollement. This means that the four rows form a ladder of height 2 in the sense of \cite{BGS,AKLY}. Similarly, we have the notion of extending the recollement one step upwards.

\smallskip

\subsubsection{Recollements of compactly generated triangulated categories}\

Let $\mathcal{T}$ and $\mathcal{T}'$ be triangulated $k$-categories with infinite direct sums and assume that they are compactly generated.

\begin{Lem}[{\cite[Theorems 4.1 and 5.1]{Neeman96}}]
\label{adjoint}
The following conditions are equivalent for an adjoint pair of triangle functors $\xymatrix{\mathcal{T}'\ar@<-.7ex>[r]_G & \mathcal{T}\ar@<-.7ex>[l]_F}$:
\begin{itemize}
\item[(a)] $F$ sends compact objects to compact objects;
\item[(b)] $G$ preserves infinite direct sums;
\item[(c)] $G$ has a right adjoint.
\end{itemize}
\end{Lem}

Let $\mathcal{T},~\mathcal{T}'$ and $\mathcal{T}''$ be triangulated $k$-categories with infinite direct sums and assume that they are compactly generated.
Assume that there is a recollement \eqref{recol:defn}. Then by Lemma~\ref{adjoint}, both $j_!$ and $i^*$ send compact objects to compact objects. By \cite[Theorem 2.1]{Neeman92a}, $i^*$ sends a set of compact generators of $\mathcal{T}$ to a set of compact generators of $\mathcal{T}'$. Moreover, by \cite[Lemma 4.2]{AKLY}, the functor $j_!$ detects compact objects, that is, if $j_!(Z)$ is compact, then $Z$ is compact.  The following is a consequence of Proposition~\ref{prop:recollement-vs-ses} and Lemma~\ref{adjoint}.

\begin{Lem}
\label{lem:extending-recollement-downward}
Assume that there is a recollement \eqref{recol:defn}. Then the following conditions are equivalent:
\begin{itemize}
\item[(i)] $i_*$ sends compact objects to compact objects,
\item[(ii)] $j^*$ sends compact objects to compact objects,
\item[(iii)] $i^!$ has a right adjoint,
\item[(iv)] $j_*$ has a right adjoint,
\item[(v)] the recollement extends one step downwards.
\end{itemize}
\end{Lem}

\smallskip

\subsubsection{Recollement generated by a homological epimorphism}\

Let $A$ and $B$ be dg $k$-algebras and $f\colon A\rightarrow B$ be a dg algebra homomorphism. Then $B$ is a dg $B$-$A$-bimodule as well as a dg $A$-$B$-bimodule via $f$, and induces a triangle functor $f^*=?\lten_B B=\RHom_B(B,?)\colon\D(B)\rightarrow\D(A)$. The following is a special case of Lemma~\ref{lem:generating-adjoint-triple}.

\begin{Cor}
\label{cor:adjoint-triple-induced-by-a-homomorphism}
There is an adjoint triple of triangle functors
\[
\xymatrixcolsep{5pc}\xymatrix{
\D(B)  \ar[r]|{f^*} &\D(A). \ar@/_1pc/[l]_{?\lten_A B} \ar@/^1pc/[l]^{\RHom_A(B,?)}
  } \]
We call this triple the adjoint triple induced by $f$.
\end{Cor}

The following definition appears right after \cite[Section 4, Lemma]{NicolasSaorin09}. For algebras, the notion was introduced in~\cite{GL91}.

\begin{Def}[{\cite[Section 4]{NicolasSaorin09}}]
The dg algebra homomorphism $f$ is called a \emph{homological epimorphism} if the triangle functor $f^*\colon\D(B)\rightarrow\D(A)$ is fully faithful.
\end{Def}

\begin{Lem}[{\cite[Section 4]{NicolasSaorin09}}]
\label{lem:recollement-induced-by-homological-epimorphism}
Assume that $f$ is a homological epimorphism. Then there is a recollement:
\[ \xymatrixcolsep{4pc}\xymatrix{\D(B) \ar[r]|{i_*} &\D(A) \ar@/_1pc/[l]_{i^*} \ar@/^1pc/[l]^{i^!} \ar[r]  &\Loc(X), \ar@/^-1pc/[l]_{j_!} \ar@/^1pc/[l]
  } \] 
where $j_!$ is the inclusion, the left half is induced by $f$,  and $X$ is defined by the triangle
\[
\xymatrix{
X\ar[r] & A\ar[r]^f & B\ar[r] & X[1].
}
\]
Assume further that $B_A\in\per(A)$ and that $X_A$ is $\K$-projective. Then there is a recollement
\begin{equation}
\label{equ:end(X)}
\xymatrixcolsep{4pc}\xymatrix{\D(B) \ar[r]|{i_*} &\D(A) \ar@/_1pc/[l]_{i^*} \ar@/^1pc/[l]^{i^!} \ar[r]  &\D(\cEnd_A(X)), \ar@/^-1pc/[l]_{j_{!=}?\lten_{\cEnd_A(X)}X} \ar@/^1pc/[l]
  }
\end{equation}
which extends one step downwards. 
\end{Lem}

\begin{proof}
The first statement appears right after the definition of homological epimorphism in \cite[Section 4]{NicolasSaorin09}. If $B_A\in\per(A)$, then $X_A\in\per(A)$. It follows from Lemma~\ref{lem:compact-object-induces-embedding} that $?\lten_{\cEnd_A(X)} X\colon\D(\cEnd_A(X))\to \D(A)$ is fully faithful with essential image $\Loc(X)$. Then the existence of the recollement follows from the first statement. Moreover, in this case $i_*$ sends compact objects to compact objects, so by Lemma~\ref{lem:extending-recollement-downward}, the recollement extends one step downwards.
\end{proof}

\smallskip

\subsubsection{Recollement for a triangular matrix algebra}\

The following generalises \cite[Example 3.4]{AKLY}.

\begin{Lem}
\label{lem:ladder-for-matrix-algebra}
Let $B$ and $C$ be dg $k$-algebras, $M$ be a dg $C$-$B$-bimodule, and let $A$ be the matrix algebra
\[
A=\left( \begin{array}{cc} B & 0\\ M & C\end{array}\right).
\]
Then there is a recollement of derived categories:
\begin{equation} \label{equ:matrix}
\xymatrixcolsep{4pc}
\xymatrix{\D(B) \ar[r]|{i_*} &\D(A) \ar@/_1pc/[l]_{i^*} \ar@/^1pc/[l]^{i^!} \ar[r]|{j^*}  &\D(C) \ar@/^-1pc/[l]_{j_!} \ar@/^1pc/[l]^{j_*}
  } \end{equation}
where the left half is induced by the cononical homomorphism $A\to B$ and the right half is induced by the dg $C$-$A$-bimodule $(M,C)$. Moreover,
\begin{itemize}
\item[(a)]
this recollement extends one step downwards;
\item[(b)] if $M_B\in\per(B)$, then it extends two steps downards;
\item[(c)] if ${}_CM\in\per(C^{\oop})$ and ${}_CM$ is $\K$-projective, then it extends one step upwards.
\end{itemize}
\end{Lem}

\begin{proof}
We first show that the diagram~\eqref{equ:matrix} is a recollement. By Proposition~\ref{prop:recollement-vs-ses}, it is enough to show that the upper two rows form a left recollement.

Let $e=\begin{pmatrix}\begin{smallmatrix} 0 & 0\\ 0 & 1_C \end{smallmatrix}\end{pmatrix}$,
 an idempotent of $A$. Then there are canonical isomorphisms $eAe\cong C$ and $A/AeA\cong B\cong (1-e)A(1-e)$ of dg $k$-algebras, $(M,C)\cong eA$ of dg $C$-$A$-bimodules, and $B\cong (1-e)A$ of dg $B$-$A$-bimodules. So by Lemma~\ref{lem:compact-object-induces-embedding}, we see that both $j_!$ and $i_*$ are fully faithful.
Since $Ae\cong eAe$ as dg $C$-modules and $eA\cong AeA$ as dg $A$-modules, the canonical morphism $Ae \lten_{eAe}eA=Ae\ten_{eAe}eA \to AeA$ is an isomorphism. Thus there is the following triangle in $\D(A^{\oop}\ten_k A)$:
\[
\xymatrix{
Ae\lten_{eAe}eA\ar[r] & A \ar[r] & B\ar[r] & Ae\lten_{eAe}eA[1]
}
\]
Thus for any dg $A$-module $M$ there is a triangle in $\D(A)$:
\[
\xymatrix@R=0.5pc{
M\lten_A (Ae\lten_{eAe}eA)\ar[r] & M \ar[r] & M\lten_A B\ar[r] & M\lten_A (Ae\lten_{eAe}eA)[1].
}
\]
It is clear that $M\lten_A B=i_*i^*(M)$. Since $Ae\lten_{eAe}eA\cong Ae\ten_{eAe}eA$, it follows that $M\lten_A (Ae\lten_{eAe}eA)\cong (M\lten_A Ae)\lten_{eAe}eA=j_!j^*(M)$.
This shows that the upper two rows of the diagram form a left recollement.

(a) We have $j^{*}(A)=Ae\cong eAe$, implying that $j^*$ sends compact objects to compact objects, so the recollement extends one step downwards
by Lemma~\ref{lem:extending-recollement-downward}.

(b) We have $i^!(A)=\RHom_A(B_A,A)\cong\left(\begin{array}{c}B\\ M\end{array}\right)\in\per(B)$, implying that $i^!$ sends compact objects to compact objects, so the recollement extends two steps downwards, by (a) and  Lemma~\ref{lem:extending-recollement-downward}.

(c) In this case, ${}_C(M,C)\in\per(C^{\oop})$ and ${}_C(M,C)$ is $\K$-projective, so $j_!$ has a left adjoint, by Lemma~\ref{lem:left-adjoint-of-tensor}. It follows that the recollement extends one step upwards, by Proposition~\ref{prop:recollement-vs-ses}.
\end{proof}

\smallskip

\subsubsection{Recollement generated by an idempotent}\

Let $A$ be non-positive dg $k$-algebra (that is, $A^p=0$ for all $p>0$), $\K$-flat as a complex of $k$-modules, and $e\in A$ be an idempotent, \ie $e^2=e$. The first statement of the following result is a generalisation of \cite[Proposition 6.1]{KalckYang18b}.

\vspace{1em}
  \begin{Prop}~\label{id}
  \label{prop:recollement-induced-by-idempotent}
  Let $A$ and $e \in A$ as above. Then there is a dg $k$-algebra $B$ with a homological epimorphism of dg $k$-algebra $f\colon A\rightarrow B$, and a recollement of derived categories:
   \[ \xymatrixcolsep{4pc}\xymatrix{\D(B) \ar[r]|{i_*} &\D(A) \ar@/_1pc/[l]_{i^*} \ar@/^1pc/[l]^{i^!} \ar[r]|{j^*}  &\D(eAe) \ar@/^-1pc/[l]_{j_!} \ar@/^1pc/[l]^{j_*}
  } \] 
such that the following conditions are satisfied:
\begin{itemize}
\item[(a)] the left half is induced by $f$, and the right half is induced by the dg $eAe$-$A$-bimodule $eA$;
\item[(b)] $B$ is a non-positive dg algebra;
\item[(c)] $\h^0(B)$ is isomorphic to $\h^0(A)/\h^0(A)e\h^0(A)$.
\end{itemize}
Moreover, if $Ae_{eAe}\in\per(eAe)$, then this recollement extends one step downwards.
\end{Prop}
\begin{proof} The proof of the first statement is similar to the proof of \cite[Proposition 6.1]{KalckYang18b}, and we omit it.
For the `moreover' part, notice that $j^*(A)=A\otimes_A^{\L} Ae=Ae$, which is compact in $\D(eAe)$ by assumption. Thus $j^*$ sends compact objects to compact objects. By Lemma \ref{lem:extending-recollement-downward}, the recollement extends one step downwards.
\end{proof}

\begin{Rem}[\cite{CPS}]\label{rem:stratifying-ideal}
Assume that $A$ is a $k$-algebra, flat as a $k$-module.
If $AeA$ is a \emph{stratifying ideal}, \ie the canonical morphism $Ae\otimes^\mathbf{L}_{eAe}{eA}\rightarrow A$ induces an isomorphism $Ae\otimes^\mathbf{L}_{eAe}{eA}\cong AeA$, then $B=A/AeA$. For example, if $k$ is a field and $A=kQ$ is the path algebra of an acyclic quiver $Q$, then $AeA$ is a stratifying ideal for any idempotent $e$ of $A$.
\end{Rem}

\smallskip

\subsubsection{Left/right recollements for $\per$, $\thick(D-)$ and $\D_{\fg}$}\

Assume that $k$ is noetherian. Let $A, B$ and $C$ be dg $k$-algebras. Consider a diagram

\[
\xymatrixcolsep{4pc}\xymatrix{\D(B) \ar@<.7ex>[r] \ar@/_1pc/[r]&\D(A) \ar@<.7ex>[l] \ar@/_1pc/[l]  \ar@<.7ex>[r]  \ar@/_1pc/[r] &\D(C). \ar@<.7ex>[l] \ar@/^-1pc/[l]}
\]
{}

\begin{Lem}
\label{lem:half-recollement-for-per-etc}
Assume that in the above diagram the upper three rows and the lower three rows are both recollements (\ie the diagram is a ladder of height 2 in the sense of \cite{BGS,AKLY}).
Then the upper two rows restrict to a left recollement of $\per(A)$ in terms of $\per(B)$ and $\per(C)$, the middle two rows restrict to a right recollement of $\D_{\fg}(A)$ in terms of $\D_{\fg}(B)$ and $\D_{\fg}(C)$. In particular, the second row restricts short exact sequences
\[
\xymatrix@R=0.3pc{
\per(B)\ar[r] & \per(A)\ar[r] & \per(C),\\
\D_{\fg}(B)\ar[r] & \D_{\fg}(A)\ar[r] & \D_{\fg}(C),
}
\]
and the third row restricts to a short exact sequence
\[
\xymatrix@R=0.3pc{
\D_{\fg}(B) & \D_{\fg}(A)\ar[l] & \D_{\fg}(C).\ar[l]
}
\]
If $k$ is a field, then the lower two rows restrict to a left recollement of $\thick(DA)$ in terms of $\thick(DB)$ and $\thick(DC)$. In particular, the third row restricts to a short exact sequence
\[
\xymatrix@R=0.3pc{
\thick(DB) & \thick(DA)\ar[l] & \thick(DC).\ar[l]
}
\]
\end{Lem}
\begin{proof}
In view of Lemma~\ref{lem:restricting-half-recollement} and Proposition~\ref{prop:recollement-vs-ses}, it is enough to show that the upper two rows restrict to $\per$, the middle two rows restrict to $\D_\fd$ and when $k$ is a field, the lower two rows restrict to $\thick(D-)$. This follows from the Lemmas~\ref{lem:Nakayama-functor} and~\ref{FkeepcompactGkeepDbmod}.
\end{proof}

\bigskip

\section{Short exact sequences of triangle quotients}
\label{ss:proof-of-Theorem:recoltoDsg}\

Let $\mathcal{T},~\mathcal{T}'$ and $\mathcal{T}''$ be triangulated $k$-categories, and $\mathcal{S},~\mathcal{S}'$ and $\mathcal{S}''$ be triangulated subcategories of $\mathcal{T},~\mathcal{T}'$ and $\mathcal{T}''$, respectively. Assume that there is a commutative diagram of triangle functors
\begin{eqnarray}
\label{ses-of-triangulated-categories}
\xymatrix{
\mathcal{S}'\ar[r]^{G|_{\mathcal{S}'}} \ar@{_{(}->}[d]& \mathcal{S}\ar[r]^{F|_{\mathcal{S}}}\ar@{_{(}->}[d] & \mathcal{S}''\ar@{_{(}->}[d]\\
\mathcal{T}'\ar[r]^{G} & \mathcal{T}\ar[r]^F & \mathcal{T}'',
}
\end{eqnarray}
where both rows are short exact sequences up to direct summands. We are interested in when the induced sequence
\begin{align}
\label{ses:quotients}
\xymatrix{
\mathcal{T}'/\mathcal{S}'\ar[r]^{\bar{G}} & \mathcal{T}/\mathcal{S}\ar[r]^{\bar{F}} & \mathcal{T}''/\mathcal{S}''
}
\end{align}
is a short exact sequence up to direct summands. By \cite[Lemma 3.2]{KalckYang16}, 
$\bar{F}$ induces a triangle equivalence $(\mathcal{T}/\mathcal{S})/(\thick\bar{G}(\mathcal{T}'/\mathcal{S}'))\simeq \mathcal{T}''/\mathcal{S}''$ up to direct summands. Moreover, $\bar{G}$ sends non-zero objects to non-zero objects. Indeed, 
Let $X\in \mathcal{T}'/\mathcal{S}'$ such that $\bar{G}(X)=0$, that is, $G(X)\in \add\mathcal{S}$. Then $F|_{\mathcal{S}}(G(X))=0$, which shows that $X\in\add\mathcal{S}'$. 
Therefore, \eqref{ses:quotients} is a short exact sequence up to direct summands if and only if $\bar{G}$ is fully faithful if and only if $\bar{G}$ is full. However, in general $\bar{G}$ may not be full, actually it may even not be faithful, see Remark~\ref{rem:counterexample} below.

\begin{Lem}\label{lem:ses-of-triangle-quotients}
If one of the following two conditions is satisfied: 
\begin{itemize}
\item[(a)] both $G|_{\mathcal{S}'}$ and $F|_{\mathcal{S}}$ admit left adjoints $G_\lambda$ and $F_\lambda$ such that $\Hom(F_\lambda(\mathcal{S}''),G(\mathcal{T}'))=0$, 
\item[(b)] both $G|_{\mathcal{S}'}$ and $F|_{\mathcal{S}}$ admit right adjoints $G_\rho$ and $F_\rho$ such that $\Hom(G(\mathcal{T}'),F_\rho(\mathcal{S}''))=0$,
\end{itemize}
then \eqref{ses:quotients} is a short exact sequence up to direct summands. If further the first row of \eqref{ses-of-triangulated-categories} is a short exact sequence, then so is \eqref{ses:quotients}.
\end{Lem}

\begin{proof}
We show that $\bar{G}$ is full.

(a)
For any $X, Y\in \mathcal{T}'/\mathcal{S}'$ and $f\in \Hom_{\mathcal{T}/\mathcal{S}}(\bar{G}(X), \bar{G}(Y))$. We may write $f$ as  $G(X) \xleftarrow{s} Z \xra{g} G(Y)$, such that
 $s$ extends to a triangle $ Z \xra{s} G(X) \xra{h} U \ra Z[1]$ in $\mathcal{T}$ with $U\in\mathcal{S}$. By Proposition~\ref{prop:recollement-vs-ses}(a),
there is a triangle $F_\lambda F(U) \ra U \xra{a} GG_\lambda(U) \ra F_\lambda F(U)[1]$ in $\mathcal{S}$. Since $G$ is fully faithful, there exists $h'\in\Hom_{\mathcal{T}'}(X,G_\lambda(U))$ such that $ah=G(h')$.
Extend $h'$ to a triangle $Z'\xra{s'} X \xra{h'} G_\lambda(U) \ra Z'[1]$ in $\mathcal{T}'$.
Now consider the octahedron
\[{\scriptsize
\begin{xy} 0;<0.35pt,0pt>:<0pt,-0.35pt>::
(300,0) *+{G(Z')[1]} ="0",
(0,160) *+{Z[1]}="1",
(420,160) *+{F_\lambda F(U)[1]}="2",
(180,240) *+{G(X)}="3",
(600,240) *+{GG_\lambda(U)[1]}="4",
(300,400) *+{U}="5",
"0", {\ar@{->>}|{-G(s')[1]} "3"}, "0", {\ar@{.>} "2"},
"1", {\ar|{t} "0"},
"1", {\ar@{->>}|{-s[1]} "3"}, "2", {\ar@{.>>} "1"},
"2", {\ar@{.>>} "5"}, "3", {\ar|{h} "5"},
"3", {\ar|(0.6){G(h')} "4"}, "4", {\ar "0"},
"4", {\ar@{.>} "2"}, "5", {\ar "1"},
"5", {\ar|{a} "4"},
\end{xy}}
\]
By the octahedral axiom, there is a triangle
\[
\xymatrix{
F_\lambda F(U)[-1] \ar[r] & Z\ar[r]^(0.45){t[-1]} & G(Z')\ar[r] & F_\lambda F(U).
}
\]
So $\mathrm{Cone}(t[-1])\cong F_\lambda F(U)\in\mathcal{S}$. Moreover, because $\Hom(F_\lambda F(U)[-1],G(Y))=0$, the morphism $g\colon Z\to G(Y)$ has to factor through $t[-1]$, so there exists $g'\in\Hom(Z',Y)$ such that $g=G(g')\circ t[-1]$.
Now the commutative diagram
\[
\xymatrix@C=4.5pc{
&Z\ar[d]|{t[-1]} \ar[ddl]|{s}\ar[ddr]|g  &\\
&G(Z')\ar[ld]|{G(s')}\ar[rd]|{G(g')}&\\
G(X) && G(Y)
}
\]
shows that $G(X) \xleftarrow{s} Z \xra{g} G(Y)$ is equivalent  to $G(X) \xleftarrow{G(s')} G(Z') \xra{G(g')} G(Y)$, the image of $X\xleftarrow{s'} Z'\xra{g'} Y$ under $\bar{G}$.

(b)
For any $X, Y\in \mathcal{T}'/\mathcal{S}'$ and $f\in \Hom_{\mathcal{T}/\mathcal{S}}(\bar{G}(X), \bar{G}(Y))$. We may write $f$ as  $G(X) \xleftarrow{s} Z \xra{g} G(Y)$, such that
 $s$ extends to a triangle $ Z \xra{s} G(X) \xra{h} U \ra Z[1]$ in $\mathcal{T}$ with $U\in\mathcal{S}$.  by Proposition~\ref{prop:recollement-vs-ses}(a), there is a triangle $G G_\rho(U) \xra{b} U \xra{a} F_\rho F(U) \ra G G_\rho(U)[1]$ in $\mathcal{S}$. Since $\Hom(G(X),F_\rho F(U))=0$ and $G$ is fully faithful, there exists $h'\in\Hom_{\mathcal{T}'}(X,G_\rho(U))$ such that $h=b\circ G(h')$.
Extend $h'$ to a triangle $Z'\xra{s'} X \xra{h'} G_\rho(U) \ra Z'[1]$ in $\mathcal{T}'$.
Now consider the octahedron
\[{\scriptsize
\begin{xy} 0;<0.35pt,0pt>:<0pt,-0.35pt>::
(300,0) *+{Z[1]} ="0",
(0,160) *+{G(Z')[1]}="1",
(420,160) *+{F_\rho F(U)}="2",
(180,240) *+{G(X)}="3",
(600,240) *+{U}="4",
(300,400) *+{GG_\rho(U)}="5",
"0", {\ar@{->>}|{-s[1]} "3"}, "0", {\ar@{.>} "2"},
"1", {\ar|{t} "0"},
"1", {\ar@{->>}|{-G(s')[1]} "3"}, "2", {\ar@{.>>} "1"},
"2", {\ar@{.>>} "5"}, "3", {\ar|{G(h')} "5"},
"3", {\ar|(0.6){h} "4"}, "4", {\ar "0"},
"4", {\ar@{.>} "2"}, "5", {\ar "1"},
"5", {\ar|{b} "4"},
\end{xy}}
\]
By the octahedral axiom, there is a triangle
\[
\xymatrix{
G(Z')\ar[r]^(0.55){t[-1]} & Z\ar[r] & F_\rho F(U) \ar[r] & G(Z')[1].
}
\]
So $\mathrm{Cone}(t[-1])\cong F_\rho F(U)[-1]\in\mathcal{S}$. Moreover, since $G$ is fully faithful, there exists $g'\in\Hom(Z',Y)$ such that $G(g')=g\circ (t[-1])$. 
Now the commutative diagram
\[
\xymatrix@C=4.5pc{
&G(Z')\ar[d]|{t[-1]} \ar[ddl]|{G(s')}\ar[ddr]|{G(g')}  &\\
&Z\ar[ld]|{s}\ar[rd]|{g}&\\
G(X) && G(Y)
}
\]
shows that $G(X) \xleftarrow{s} Z \xra{g} G(Y)$ is equivalent  to $G(X) \xleftarrow{G(s')} G(Z') \xra{G(g')} G(Y)$, the image of $X\xleftarrow{s'} Z'\xra{g'} Y$ under $\bar{G}$.
\end{proof}

We recover \cite[Lemma 2.3]{Lu17}.
\begin{Cor}
\label{cor:half-recollement-of-quotient}
Let
\begin{align*}
\xymatrixcolsep{4pc}\xymatrix{\mathcal{T'} \ar[r]|{G} &\mathcal{T} \ar@/_1pc/[l]_{G_\lambda}  \ar[r]|{F}  &\mathcal{T''} \ar@/^-1pc/[l]_{F_\lambda}
  }
\end{align*}
be a left recollement of $\mathcal{T}$ in terms of $\mathcal{T}'$ and $\mathcal{T}''$ which restricts to a left recollement of $\mathcal{S}$ in terms of $\mathcal{S}'$ and $\mathcal{S}''$.
Then it induces a left recollement
\begin{align*}
\xymatrixcolsep{4pc}\xymatrix{\mathcal{T'}/\mathcal{S}' \ar[r]|{\bar{G}} &\mathcal{T}/\mathcal{S} \ar@/_1pc/[l]_{\bar{G}_\lambda}  \ar[r]|{\bar{F}}  &\mathcal{T''}/\mathcal{S}''. \ar@/^-1pc/[l]_{\bar{F}_\lambda}
  }
\end{align*}
\end{Cor}
\begin{proof}
Since $\Hom(F_\lambda(\mathcal{S}''),G(\mathcal{T}))=\Hom(\mathcal{S}'',FG(\mathcal{T}))=0$, it follows from Lemma~\ref{lem:ses-of-triangle-quotients} that the second row of the above diagram is a short exact sequence. The desired result then follows from Proposition~\ref{prop:recollement-vs-ses}(a) and \cite[Lemma 1.1]{Orlov04}.
\end{proof}

We end this section with a remark.

\begin{Rem}
\label{rem:counterexample}
The sequence \eqref{ses:quotients} is in general not a short exact sequence up to direct summands. 
We give a class of examples. Let $Q$ be an acyclic quiver and let $i$ be a vertex of $Q$ which is neither a source nor a sink. Let $Q^1$ be the quiver obtained from $Q$ by removing $i$ and all arrows incident to $i$, and let $Q^2$ be the quiver obtained from $Q^1$ by adding a new arrow $[\alpha\beta]: s(\beta)\to t(\alpha)$ for every pair $(\alpha,\beta)$ of arrows of $Q$ with $t(\beta)=i=s(\beta)$. Here an arrow $\rho$ starts at $s(\rho)$ and ends at $t(\rho)$. Since $i$ is neither a source nor a sink, it follows that $Q^1$ and $Q^2$ are different.

Now let $A=kQ$ be the path algebra of $Q$ and let $e=e_i$ be the trivial path at $i$. Then $eAe\cong k$, $A/AeA=kQ^1$,  $(1-e)A(1-e)=kQ^2$ and $A/A(1-e)A\cong k$. By Remark~\ref{rem:stratifying-ideal} and Lemma~\ref{lem:half-recollement-for-per-etc}, there are short exact sequences
\[
\xymatrix@C=5pc@R=0.8pc{
\K^b(\proj (1-e)A(1-e))\ar[r]^(0.6){?\lten_{(1-e)A(1-e)} (1-e)A} & \K^b(\proj A)\ar[r]^(0.45){?\lten_A A/A(1-e)A} & \K^b(\proj A/A(1-e)A),\\
\K^b(\proj eAe)\ar[r]^{?\lten_{eAe} eA} & \K^b(\proj A)\ar[r]^{?\lten_A A/AeA} & \K^b(\proj A/AeA).
}
\]
It is straightforward to check that the composition 
\[
\xymatrix@C=5pc{\K^b(\proj eAe)\ar[r]^{?\lten_{eAe} eA}&\K^b(\proj A)\ar[r]^(0.45){?\lten_A A/A(1-e)A} & \K^b(\proj A/A(1-e)A)}
\] 
is a triangle equivalence and that the following diagram is commutative:
\[
\xymatrix@C=5pc@R=2pc{
0\ar[r]\ar[d]& \K^b(\proj eAe)\ar[d]^{?\lten_{eAe} eA} \ar[r]^(0.45){\simeq} & \K^b(\proj A/A(1-e)A)\ar@{=}[d]\\
\K^b(\proj (1-e)A(1-e))\ar[r]^(0.6){?\lten_{(1-e)A(1-e)} (1-e)A}\ar@{=}[d] & \K^b(\proj A)\ar[r]^(0.45){?\lten_A A/A(1-e)A} \ar[d]^{?\lten_A A/AeA} &  \K^b(\proj A/A(1-e)A)\ar[d]\\
\K^b(\proj (1-e)A(1-e))\ar[r] & \K^b(\proj A/AeA) \ar[r] & 0,
}
\]
where the left functor in the third row is induced by the algebra homomorphism $(1-e)A(1-e)=kQ^2\to kQ^1=A/AeA$ of killing the $[\alpha\beta]$'s, which is neither full nor faithful.

\end{Rem}

\section{A localisation theorem for the singularity category}

In this section, we consider the short exact sequence of singularity categories induced by a recollement of derived categories.
\smallskip

Assume that $k$ is noetheiran. A dg $k$-algebra $A$ is said to be \emph{proper} if $\h^*(A)$ is finitely generated over $k$, \ie $\per(A)\subset \D_{\fg}(A)$. Finite-dimensional $k$-algebras (when $k$ is a field) are typical examples of proper dg $k$-algebras.
The \emph{singularity category} of a proper dg $k$-algebra $A$ is defined as the triangle quotient
 \[
 \D_{\sg}(A):=\D_{\fg}(A)/\per(A).
 \]
By Lemma~\ref{lem:half-recollement-for-per-etc} and Lemma~\ref{lem:ses-of-triangle-quotients}, we obtain the following theorem, which is the main result of this paper.

\begin{Thm}\label{Thm:recoltoDsg}
Let $A$, $B$ and $C$ be proper dg $k$-algebras together with a recollement
\begin{align}
\label{recol:dg-algebras-unbounded}
\xymatrixcolsep{4pc}\xymatrix{\D(B) \ar[r]|{i_*} &\D(A) \ar@/_1pc/[l]|{i^*} \ar@/^1pc/[l]|{i^!} \ar[r]|{j^*}  &\D(C). \ar@/^-1pc/[l]|{j_!} \ar@/^1pc/[l]|{j_*}}
\end{align}
Assume that the recollement extends one step downwards. Then the second row induces a short exact sequence
\begin{equation}
\xymatrix{
 \D_{\sg}(B) \ar[r]^{\bar{i}_{*}} & \D_{\sg}(A) \ar[r]^{\bar{j}^{*}} & \D_{\sg}(C).
 }\label{equ:sesDsg}
\end{equation}
 \end{Thm}

\medskip

\subsection{Finiteness of global dimension}\

For a proper dg $k$-algebra $A$, we say that $A$ has \emph{finite global dimension} and write \emph{$\gl A<\infty$} if $\per A=\D_{\fg}(A)$. Part (b) of the following result generalises \cite[Theorem 4]{Qin20}, and part (c) generalises \cite[Lemma 2.1]{Wiedemann91} (see also \cite[Corollary 5]{Koenig91}, \cite{Happel93} and \cite[Proposition 2.14]{AKLY}).

\begin{Cor}\label{Cor:gldim}
Let $A$, $B$ and $C$ be proper dg $k$-algebras together with a recollement \eqref{recol:dg-algebras-unbounded}. Then
\begin{itemize}
\item[(a)] $i_*$ induces a triangle equivalence $\D_{\sg}(B) \simeq\D_{\sg}(A)$ if and only if $\gl C<\infty$.
\item[(b)] $j^*$ induces a triangle equivalence $\D_{\sg}(A)\simeq \D_{\sg}(C)$ if and only if $i_*(B)\in \per(A)$ and $\gl B<\infty$.
\item[(c)] $\gl A<\infty$ if and only if $\gl B<\infty$ and $\gl C<\infty$.
\end{itemize}
\end{Cor}
\begin{proof}
(a) Assume that $\gl C<\infty$. Then $\per(C)=\D_{\fg}(C)$ and $\D_{\sg}(C)=0$. Since $A$ is proper and $j^{*}(\D_{\fg}(A))\subset \D_{\fg}(C)$, we have $j^{*}(\per A)\subset \per C$. By Lemma~\ref{lem:extending-recollement-downward}, the recollement extends one step downwards.  By Theorem \ref{Thm:recoltoDsg}, there is a short exact sequence \eqref{equ:sesDsg}, and then $i_*$ induces a triangle equivalence $\D_{\sg}(B)\simeq \D_{\sg}(A)$.

Assume that $i_*$ induces a triangle equivalence $\D_{\sg}(B)\simeq \D_{\sg}(A)$. Then it is necessary that $i_*(\per(B))\subset \per(A)$. By Lemma~\ref{lem:extending-recollement-downward}, the recollement extends one step downwards. So by Theorem \ref{Thm:recoltoDsg}, there is a short exact sequence \eqref{equ:sesDsg}, forcing $\D_{\sg}(C)$ to be trivial, \ie $\gl C<\infty$.

(b) Assume that $i_*(B)\in \per(A)$ and $\gl B<\infty$. Then $i_*$ sends compact objects to compact objects, and by Lemma~\ref{lem:extending-recollement-downward}, the recollement extends one step downwards. So by Theorem \ref{Thm:recoltoDsg}, there is a short exact sequence \eqref{equ:sesDsg}. But $\D_{\sg}(B)=0$, so $j^*$ induces a triangle equivalence $\D_{\sg}(A)\simeq \D_{\sg}(C)$.

Assume that $j^*$ induces a triangle equivalence $\D_{\sg}(A)\simeq \D_{\sg}(C)$. As in the proof of \cite[Theorem 4]{Qin20}, we can show that $i_*(B)\in\per(A)$. Then by Lemma~\ref{lem:extending-recollement-downward}, the recollement extends one step downwards, and by Theorem \ref{Thm:recoltoDsg}, there is a short exact sequence \eqref{equ:sesDsg}. This forces $\D_{\sg}(B)$ to be trivial, \ie $\gl B<\infty$.

(c) Assume $\gl A<\infty$. Then $\per A=\D_{\fg}(A)$ and $\D_{\sg}(A)=0$. Since $i_{*}(\D_{\fg}(B))\subset \D_{\fg}(A)$ and $B$ is proper, we know $i_{*}(\per B)\subset \per A$. By Lemma~\ref{lem:extending-recollement-downward}, the recollement extends one step downwards. By Theorem \ref{Thm:recoltoDsg}, there is a short exact sequence \eqref{equ:sesDsg}, and then $\D_{\sg}(B)=0=\D_{\sg}(C)$. So $\gl B<\infty$ and $\gl C<\infty$.

Conversely, assume $\gl B<\infty$ and $\gl C<\infty$. Then by (a) we have $\D_{\sg}(A) \simeq \D_{\sg}(B)$, which is trivial. So $\gl A<\infty$.
\end{proof}

\begin{Rem}
We give a concrete example that $\gl B<\infty$, $i_*(B)\not\in\per(A)$, and $\D_{\sg}(A) \not\simeq \D_{\sg}(C)$. Assume that $k$ is a field and let $A$ be the $k$-algebra with radical square zero associated to the quiver
\[
\xymatrix{
1\ar[r]^\beta \ar@(dl,ul) & 2\ar[r]^\gamma & 3 \ar@(dr,ur)
}
\]
Let $e=e_1+e_3$, the sum of the trivial paths at the vertices $1$ and $3$. Then $eAe\cong k[x]/(x^2)\times k[y]/(y^2)$, $A/AeA\cong k$, and as an $A$-module $A/AeA$ is the simple module
$S_2$ at the vertex $2$. It is easy to see that $\mathrm{Ext}^p_A(S_2,S_2)=0$ for all $p>0$. It follows that $A\to A/AeA$ is a homological epimorphism and hence there is a recollement
\vspace{10pt}
\[
\xymatrixcolsep{4pc}
\xymatrix{
\D(k) \ar[r] &\D(A) \ar@/_1pc/[l] \ar@/^1pc/[l] \ar[r]  &\D(k[x]/(x^2)\times k[y]/(y^2)). \ar@/^-1pc/[l] \ar@/^1pc/[l]
}
\]
\smallskip

\noindent We claim that $\D_{\sg}(A)$ is Hom-infinite. Then it can not be equivalent to $\D(k[x]/(x^2)\times k[y]/(y^2))$, which is equivalent to $\mod k\times \mod k$. To prove the claim, consider the short exact sequence
\[
\xymatrix{
0\ar[r] & S_1\oplus S_2 \ar[r] & e_1A\ar[r] & S_1\ar[r] & 0.
}
\]
This implies that $S_1\cong S_1[1]\oplus S_2[1]$ in $\D_{\sg}(A)$. This is not possible in a Hom-finite triangulated category, because $S_2$ is not zero in $\D_{\sg}(A)$ (it has infinite projective dimension over $A$).
\end{Rem}

\medskip

\subsection{Recollements of singularity categories}\

When $k$ is a field and $A$, $B$ and $C$ are finite-dimensional $k$-algebras, the first statement of the following result is \cite[Proposition 3]{Qin20}.

\begin{Cor}
\label{cor:recollement-of-singularity-categories}
Let $A$, $B$ and $C$ be proper dg $k$-algebras together with a recollement \eqref{recol:dg-algebras-unbounded}. If the recollement extends two steps downwards, then there is a right recollement of $\D_{\sg}(A)$ by $\D_{\sg}(B)$ and $\D_{\sg}(C)$. If it extends three steps downwards, then there is a recollement of $\D_{\sg}(A)$ by $\D_{\sg}(B)$ and $\D_{\sg}(C)$.
\end{Cor}
\begin{proof}
This follows immediately from Theorem~\ref{Thm:recoltoDsg} and \cite[Lemma 1.1]{Orlov04}.
\end{proof}

\medskip

\subsection{The singularity category of a triangular matrix algebra}\

When $k$ is a field and $A$ is a finite-dimensional $k$-algebra, the third statement of the following result is exactly \cite[Theorem 3.2]{LiuLu15}.

\begin{Cor}
Let $B$ and $C$ be proper dg $k$-algebras and $M$ be a dg $C$-$B$-bimodule with finite-dimensional total cohomology. Consider the matrix dg algebra $A=\left( \begin{array}{cc} B & 0\\ M & C\end{array}\right)$. Then there is a short exact sequence of triangulated categories
 \[
 \xymatrix{
 \D_{\sg}(B) \ar[r] &  \D_{\sg}(A) \ar[r] & \D_{\sg}(C).}
 \]
Moreover,
\begin{itemize}
\item[(a)]
If $M_B\in\per(B)$, then there is a right recollement of $\D_{\sg}(A)$ by $\D_{\sg}(B)$ and $\D_{\sg}(C)$.
\item[(b)]
If ${}_{C}M\in\per(C^{\oop})$, then there is a left recollement of $\D_{\sg}(A)$ by $\D_{\sg}(B)$ and $\D_{\sg}(C)$.
\item[(c)]
Assume that $B$ is $\K$-projective over $k$. If $M_B\in\per(B)$ and ${}_C M\in\per(C^{\oop})$, then there is a recollement of $\D_{\sg}(A)$ by $\D_{\sg}(B)$ and $\D_{\sg}(C)$.
\end{itemize}
\end{Cor}
\begin{proof}
The first statement follows from Lemma~\ref{lem:ladder-for-matrix-algebra}(a) and Theorem~\ref{Thm:recoltoDsg}. Statement (a) follows from Lemma~\ref{lem:ladder-for-matrix-algebra}(b) and Corollary~\ref{cor:recollement-of-singularity-categories}. To prove (b),
let $M'$ be a $\K$-projective resolution of $M$ over $C^{\oop}\ten_k B$ and let $A'=\left( \begin{array}{cc} B & 0\\ M' & C\end{array}\right)$. Then $A$ is quasi-isomorphic to $A'$, so $\D_{\sg}(A)$ is triangle equivalent to $\D_{\sg}(A')$. Moreover, ${}_C M'\in\per(C^{\oop})$ and ${}_C M'$ is $\K$-projective. Therefore by Lemma~\ref{lem:ladder-for-matrix-algebra}(c) and Corollary~\ref{cor:recollement-of-singularity-categories}, there is a left recollement of $\D_{\sg}(A')$ by $\D_{\sg}(B)$ and $\D_{\sg}(C)$. Replacing $\D_{\sg}(A')$ by $\D_{\sg}(A)$, we obtain (b). Finally, (c) follows from (a) and (b), because the lower row of the left recollement in (b) is exactly the upper row of the right recollement in (a).
\end{proof}

\medskip

\subsection{Singular reduction with respect to idempotents}
\label{ss:singular-reduction-wrt-idem}\

Let $A$ be a non-positive proper dg $k$-algebra, and $e\in A$ be an idempotent. The following is an immediate corollary of Theorem \ref{Thm:recoltoDsg}.

\begin{Cor}\label{prop5.3}
Let $B$ be as in Proposition~\ref{prop:recollement-induced-by-idempotent} and assume that $Ae_{eAe}\in\per(eAe)$. Then there is a short exact sequence of triangulated categories:
\[
\xymatrix{\D_{\sg}(B)\ar[r]^{\w{i_*}} & \D_{\sg}(A) \ar[r]^{\w{j^*}} & \D_{\sg}(eAe).}
\]
\end{Cor}

\begin{Cor}[{\cite[Theorem 3.1]{Chen10} and \cite[Theorem 2.1]{Chen09}}]
Assume that $A$ is a finite-dimensional $k$-algebra. Assume that $\mathrm{p.d.}{Ae}_{eAe}<\infty$. Then the \emph{Schur} functor $S_e= ?\otimes_A Ae\colon \mod A\rightarrow \mod eAe$ induces a triangle equivalence
\[
\D_{\sg}(A)/\thick\langle q(\mathcal{B}_e)\rangle \simeq \D_{\sg} (eAe),
\]
where $q\colon\D_{\fd}(A)\rightarrow\D_{\sg}(A)$ is the natural quotient functor and $\mathcal{B}_e$ is the essential kernel of $S_e$. Moreover, if in addition every finitely generated \emph{A/AeA}-module has finite projective dimension over $A$, then there is a singular equivalence $\D_{\sg}(A)\stackrel{\simeq}{\longrightarrow}\D_{\sg}(eAe)$.
\end{Cor}
Note that \cite[Theorem 3.1]{Chen10} and \cite[Theorem 2.1]{Chen09} are stated more generally for left Noetherian rings.
\begin{proof}
First note that $\thick(\mathcal{B}_e)$ is exactly $\ke (j^*\colon \D_{\fd}(A)\to\D_{\fd}(eAe))$ (\cite[lemma 2.2]{Chen09}), thus it is triangle equivalent to $\D_{\fd}(B)$, by Lemma~\ref{lem:half-recollement-for-per-etc}. So $\thick\langle q(\mathcal{B}_e)\rangle$ is triangle equivalent to $\D_{\sg}(B)$ up to taking direct summands. Then the first statement follows from Corollary \ref{prop5.3}.
The second statement follows from the first one, because
 under the additional assumption,
 $i_{*}(\D_{\fd}(B))\subset \per(A)$, and hence, $\D_{\fd}(B)=i^{*}i_{*}(\D_{\fd}(B))\subset \per(B)$. Therefore $\D_{\sg}(B)=0$.
\end{proof}

\medskip

\subsection{Application to homological epimorphisms}\

In view of Lemma~\ref{lem:recollement-induced-by-homological-epimorphism}, the following result is an immediate consequence of Theorem~\ref{Thm:recoltoDsg}.

\begin{Cor}
\label{cor:s.e.s.-of-Dsg-homological-epi}
Let $f\colon A\rightarrow B$ be a homological epimorphism of proper dg $k$-algebras. Assume $B_A\in\per(A)$. Then there is a short exact sequence of triangulated categories:
\[
\xymatrix{
\D_{\sg}(B)\ar[r] & \D_{\sg}(A) \ar[r] & \D_{\sg}(\cEnd_A(X)),
}
\]
where $X$ is a $\K$-projective resolution of $\mathrm{Cone}(f)[-1]$.
\end{Cor}

\begin{Cor}[{\cite[Theorem]{Chen14}}]
Assume that $k$ is a field. Let $A$ be a finite-dimensional $k$-algebra and $J$ be an ideal of $A$ such that the canonical homomorphism $A\to A/J$ is a homological epimorphism. Assume that $J$ has finite projective dimension as an $A$-$A$-bimodule. Then there is a triangle equivalence $\D_{\sg}(A)\stackrel{\simeq}\to \D_{\sg}(A/J)$.
\end{Cor}
\begin{proof}
Let $B=A/J$. Since $J$ has finite projective dimension as a right $A$-module, it follows that $B_A\in\per(A)$. By Corollary~\ref{cor:s.e.s.-of-Dsg-homological-epi}, there is a short exact sequence
\[
\xymatrix{
\D_{\sg}(B)\ar[r] & \D_{\sg}(A) \ar[r] & \D_{\sg}(\cEnd_A(\mathbf{p}J)).
}
\]
Here $\mathbf{p}J$ is a projective resolution of $J$ over $A$. It is enough to show $\D_{\sg}(\cEnd_{A}(\mathbf{p}J))$=0.
Since $\mathbf{p}J\in \per(A)$, it follows that $\cEnd_{A}(\mathbf{p}J)$ is proper. For any $M\in \D_{\fd}(\cEnd_{A}(\mathbf{p}J))$, we only need to check that $M \in  \per(\cEnd_{A}(\mathbf{p}J))$. Consider the induced recollement \eqref{equ:end(X)}. The middle row restricts to a short exact sequence of $\D_{\fd}$ by Lemma~\ref{lem:half-recollement-for-per-etc}. In particular,  there exists $N\in \D_{\fd}(A)$ such that $j^{*}(N)=M$.
Applying $N\lten_{A}?$ to the triangle $J\ra A \ra B \ra J[1]$, we obtain a triangle $N\lten_{A}J \ra N \ra N\lten_{A}B \ra N\lten_AJ[1]$ in $\D(A)$.  Note that the morphism $N\ra N\lten_{A}B$ is induced by $f$, so $N\lten_{A}J \cong j_{!}j^{*}(N)$ by \eqref{equ:end(X)}. Also note that $N\lten_{A}J\in \per(A)$ because $J\in\per(A^{\oop}\ten_k A)$. Then $M\cong j^{*}j_{!}(M)\cong j^{*}j_{!}j^{*}(N)\in \per(\cEnd_{A}(\mathbf{p}J))$, because $j^{*}$ sends compact objects to compact objects.
\end{proof}

\medskip

\subsection{A dual statement}
The following is a `dual' of Theorem~\ref{Thm:recoltoDsg}.  

\begin{Thm}\label{Thm:recoltoDsg-dual}
Assume that $k$ is a field. Let $A$, $B$ and $C$ be proper dg $k$-algebras together with a recollement \eqref{recol:dg-algebras-unbounded}. Assume that the recollement extends one step upwards. Then the second row induces a short exact sequence
\begin{align}
\xymatrix{
 \D_{\fd}(B)/\thick(DB) \ar[r]^{\bar{i}_{*}}& \D_{\fd}(A)/\thick(DA)  \ar[r]^{\bar{j}^{*}} &  \D_{\fd}(C)/\thick(DC).
 }\label{equ:sesDD}
 \end{align}
 \end{Thm}
 \begin{proof}
 This follows from Lemma~\ref{lem:half-recollement-for-per-etc} and Lemma~\ref{lem:ses-of-triangle-quotients}.
 \end{proof}

%

\bigskip

\section{A localisation theorem for the AGK-category}
\label{section:cluster}\

In this section, we establish a result for homologically smooth dg algebras, which is `Koszul dual' to Theorem \ref{Thm:recoltoDsg}.

\medskip

Assume that $k$ is a field and let $A$ be a dg $k$-algebra. Assume that $A$ is homologically smooth, \ie $A\in \per(A^{\rm e})$, where $A^{\rm e}=A^{\oop}\ten_k A$. Then it is shown in the proof of \cite[Lemma 4.1]{Keller10} that $\D_{\fd}(A)\subset \per(A)$. We define the \emph{AGK-category} of $A$ as the triangle quotient
\[
\D_{\agk}(A):=\per(A)/\D_{\fd}(A).
\]
Note that if $A$ satisfies the following conditions
 \begin{enumerate}[\rm(1)]
 \item $A$ is non-positive;
 \item $A$ is homologically smooth;
 \item $\h^{0}(A)$ is finite-dimensional;
 \item   $A$ is \emph{bimodule $(d+1)$-Calabi-Yau} for $d\ge 1$, $i.e.$  there is an isomorphism in $\D(A^{\rm e})$
 \[ \RHom_{A^{\rm e}}(A, A^{\rm e})\simeq A[-d-1]. \]
 \end{enumerate}
Then the AGK-category of $A$ is the \emph{cluster category} $\C_{A}$ (\cite{Amiot09,Guo11}).

\medskip

The main result of this section is

\begin{Thm}\label{Thm:sescluster}
Let $A$, $B$ and $C$ be homologically smooth dg $k$-algebras together with a recollement \eqref{recol:dg-algebras-unbounded} which extends one step downwards. Then the second row induces a short exact sequence
\begin{equation}\label{equ:sesAGK}
\xymatrix{
 \D_{\agk}(B) \ar[r]^{\ov{i_{*}}} & \D_{\agk}(A) \ar[r]^{\ov{j^{*}}} & \D_{\agk}(C).
 }
\end{equation}
\end{Thm}

\begin{proof}
This follows from Lemma~\ref{lem:half-recollement-for-per-etc} and Lemma~\ref{lem:ses-of-triangle-quotients}.
\end{proof}


\smallskip

Similar to Corollary \ref{Cor:gldim}, we have
\begin{Cor}\label{Cor:per/Dfd}
Let $A$, $B$ and $C$ be homologically smooth dg $k$-algebras together with a recollement \eqref{recol:dg-algebras-unbounded}. Then
\begin{itemize}
\item[(a)] $j^*$ induces a triangle equivalence $\D_{\agk}(A) \simeq \D_{\agk}(C)$ if and only if $B$ is proper .
\item[(b)] $i_*$ induces a triangle equivalence $\D_{\agk}(B)\simeq \D_{\agk}(A)$ if and only if $j^*(A)\in\per(C)$ and $C$ is proper.
\item[(c)] $A$ is proper if and only if $B$ and $C$ are proper.
\end{itemize}
\end{Cor}
\begin{proof}
(a) Assume that $B$ is proper. Then $\per(B)=\D_{\fd}(B)$, so
\[
i_*(\per(B))=i_*(\D_{\fd}(B))\subset \D_{\fd}(A)\subset \per(A),
\]
namely, $i_*$ sends compact objects to compact objects. By Lemma~\ref{lem:extending-recollement-downward}, the given recollement extends one step downwards. But $\D_{\agk}(B)=0$, which, together with the short exact sequence in Theorem~\ref{Thm:sescluster}, implies that $j^*$ induces a triangle equivalence $\D_{\agk}(A) \simeq \D_{\agk}(C)$.

Assume that $j^*$ induces a triangle equivalence $\D_{\agk}(A)\simeq \D_{\agk}(C)$. Then $j^*$ necessarily restricts to $\per$, so the recollement extends one step downwards by Lemma~\ref{lem:extending-recollement-downward}. Therefore by Theorem~\ref{Thm:sescluster}, there is a short exact sequence \eqref{equ:sesAGK}, implying that $\D_{\agk}(B)=0$, \ie $B$ is proper.

(b) Assume that $i_*$ induces a triangle equivalence $\D_{\agk}(B)\simeq \D_{\agk}(A)$. Then necessarily $i_*$ restricts to $\per$, so by Lemma~\ref{lem:extending-recollement-downward}, $j^*(A)\in\per(C)$ and the recollement extends one step downwards. By Theorem~\ref{Thm:sescluster}, there is a short exact sequence \eqref{equ:sesAGK}, forcing $\D_{\agk}(C)$ to be trivial, \ie $C$ is proper.

Assume that $j^*(A)\in\per(C)$ and $C$ is proper, \ie $\D_{\agk}(C)=0$. Then $j^*$ sends compact objects to compact objects, so by Lemma~\ref{lem:extending-recollement-downward} the recollement extends one step downwards, and by  Theorem~\ref{Thm:sescluster}, there is a short exact sequence \eqref{equ:sesAGK}, showing that $i_*$ induces a triangle equivalence $\D_{\agk}(B)\simeq \D_{\agk}(A)$.

(c) If $B$ and $C$ are proper, then by (a) we have $\D_{\agk}(A)\simeq\D_{\agk}(C)\simeq 0$. If $A$ is proper, then $C$ is also proper, because  $C\cong j^{*}j_{!}(C)\subset \D_{\fd}(C)$.
Moreover, there is a triangle in $\D(A)$
\[
\xymatrix{
i_*i^*(A)\ar[r] & A \ar[r] & j_!j^*(A) \ar[r] & i_*i^*(A)[1].
}
\]
Since both $A$ and $C$ are proper, it follows that both $A$ and $j_!j^*(A)$ belong to $\D_{\fd}(A)$, so does $i_*i^*(A)$. Therefore $i^*(A)$ belongs to $\D_{\fd}(B)$, by Lemma~\ref{FkeepcompactGkeepDbmod} and the fact that $i^*(A)$ is a classical generator of $\per(B)$. So $\per(B)\subset \D_{\fd}(B)$, that is, $B$ is proper.
\end{proof}

\begin{Ex}\label{Ex:cycle}
Let $A$ be the graded path algebra of a graded cyclic quiver
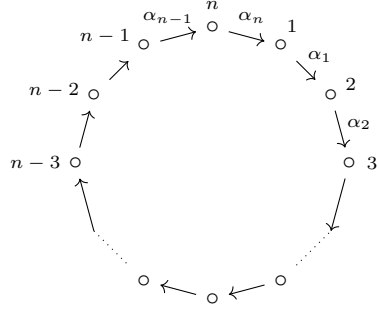
\begin{center}
 \begin{tikzpicture}
 \node (1) at (60:1.8) {$\circ$};
 \node (2) at (30:1.8) {$\circ$};
 \node (3) at (360:1.8) {$\circ$};
 \node (5) at (300:1.8) {$\circ$};
 \node (6) at (270:1.8) {$\circ$};
 \node (7) at (240:1.8) {$\circ$};
 \node (9) at (180:1.8) {$\circ$};
 \node (10) at (150:1.8) {$\circ$};
 \node (11) at (120:1.8) {$\circ$};
 \node (12) at (90:1.8) {$\circ$};
 \node at (45:2) {\scriptsize$\alpha_{1}$};
 \node at (15:2) {\scriptsize$\alpha_{2}$};
 \node at (75:2) {\scriptsize$\alpha_{n}$};
 \node at (107:2.0) { \scriptsize$\alpha_{n-1}$};
 \draw[<-] (2) -- (1);
 \draw[<-] (1) -- (12);
 \draw[<-] (3) -- (2);
 \draw[<-] (330:1.8) -- (3);
 \draw[dotted] (5) -- (330:1.8);
 \draw[<-] (6) -- (5);
 \draw[<-] (7) -- (6);
 \draw[dotted] (210:1.8) -- (7);
 \draw[<-] (9) -- (210:1.8);
 \draw[<-] (10) to(9);
 \draw[<-] (11) -- (10);
 \draw[<-] (12) -- (11);

 \node at (60:2.1) {\scriptsize$1$};
 \node at (30:2.1) {\scriptsize$2$};
 \node at (360:2.1) {\scriptsize$3$};
 \node at (90:2.1) {\scriptsize$n$};
 \node at (130:2.2) {\scriptsize $n-1$};
 \node at (155:2.3) {\scriptsize$n-2$};
 \node[left=2pt] at (9) {\scriptsize$n-3$};
\end{tikzpicture}
\end{center}
Let $d=\sum_{i=1}^{n}|\za_{i}|$.  Let $C=e_{1}Ae_{1}$ and $B=A/Ae_1A$. Then $A$, $B$ and $C$ are all homologically smooth and moreover, there is a recollement  \eqref{recol:dg-algebras-unbounded} by \cite[Lemma 7.2]{KalckYang18b}. Notice that $B$ is the graded path algebra of a graded quiver of type $\mathbb{A}_{n-1}$ with linear orientation, in particular, it is proper. So
by Corollary \ref{Cor:per/Dfd}, there a triangle equivalence
\[ \D_{\agk}(A)  \simeq \D_{\agk}(k[x]),  \]
where $|x|:=d$. If $d=0$, then $\per(k[x])\simeq\K^b(\proj k[x])$ and $\D_{\agk}(k[x])\simeq \D^b(\mod k)$. 
Next, assume that $d\neq 0$. Then $\per(k[x])$ is a Krull--Schmidt category and as in \cite[Theorem 4.1]{KellerYangZhou09}, a complete set of indecomposable objects of $\per(k[x])$ is given by
\[
k[x][p],~~(k[x]/(x^n))[p],~~n\in\mathbb{N},~p\in\mathbb{Z},
\]
and a complete set of indecomposable objects of $\D_{\fd}(k[x])$ is given by
\[
(k[x]/(x^n))[p],~~n\in\mathbb{N},~p\in\mathbb{Z}.
\]
Moreover,
\[
\Hom_{\per(k[x])}(k[x],k[x][p])=\h^p(k[x])=\begin{cases} k & \text{if $p\in\mathbb{N}$ is a multiple of $d$,}\\ 0 & \text{otherwise.}\end{cases}
\]
Direct computation shows that $\D_{\agk}(k[x])$ is a semisimple category with simple objects $k[x][p]$, $p=0,\ldots,|d|-1$. To summarise, there are triangle equivalences
\[  \D_{\agk}(A)  \simeq \D_{\agk}(k[x])\simeq \D^{\rm b}({\rm mod} k)/[d].\]
\end{Ex}

\bigskip

\section{A generalisation of Theorem \ref{Thm:recoltoDsg}}

Let  $\mathcal{T}$, $\mathcal{T}'$ and $\mathcal{T}''$ be compactly generated triangulated $k$-categories together with a recollement
\begin{align}
\label{recol:abstract-cat}
\xymatrixcolsep{4pc}
\xymatrix{
\mathcal{T}' \ar[r]|-{i_*} &\mathcal{T} \ar@/_1pc/[l]|-{i^*} \ar@/^1pc/[l]|-{i^!} \ar[r]|-{j^*}  &\mathcal{T}''. \ar@/^-1pc/[l]|-{j_!} \ar@/^1pc/[l]|-{j_*}  }
\end{align}
Assume that $\mathcal{T}$ is algebraic.
Let $\cal S$, $\cal S'$ and $\cal S''$ be skeletally small triangulated subcategories of $\mathcal{T}$, $\mathcal{T}'$ and $\mathcal{T}''$, which   contain $\mathcal{T}^c$, $\mathcal{T}'^c$ and $\mathcal{T}''^c$, respectively.

The main result of this section is:
\begin{Thm}
\label{thm:s.e.s.-of-quotients-in-the-first-row}
Assume further that the first row of the above recollement restricts to a short exact sequence (respectively, a short exact sequence up to direct summands)
\[
\xymatrix{
\cal S' & \cal S\ar[l]_{i^*} & \cal S''\ar[l]_{j_!}.
}
\]
Then there is a short exact sequence (respectively, a short exact sequence up to direct summands):
\[
\xymatrix{
\mathcal{S}'/\mathcal{T}'^c &  \mathcal{S}/\mathcal{T}^c \ar[l]_{\bar{i}^*} & \mathcal{S}''/\mathcal{T}''^c \ar[l]_{\bar{j}_!}.
}
\]
\end{Thm}

\medskip

\subsection{Neeman's localisation theorem}\

The following is a well-known reformulation of Neeman's localisation theorem \cite[Theorem  2.1]{Neeman92a}.

\begin{Thm}[{\cite[Theorem  2.1]{Neeman92a}}]
\label{thm:Neeman's-localisation-theorem}
Assume that there is a recollement of compactly generated triangulated $k$-categories:
\[
 \xymatrixcolsep{4pc}\xymatrix{\mathcal{T'} \ar[r]|{i_*} &\mathcal{T} \ar@/_1pc/[l]|{i^*} \ar@/^1pc/[l]|{i^!} \ar[r]|{j^*}  &\mathcal{T''}. \ar@/^-1pc/[l]|{j_!} \ar@/^1pc/[l]|{j_*}
  }
\]
Then the first row restricts to a short exact sequence up to  direct summands
\[
\xymatrix{
\mathcal{T'}^c & \mathcal{T}^c \ar[l]_{i^*} & \mathcal{T''}^c\ar[l]_{j_!}.
}
\]
\end{Thm}

We have the following converse.

\begin{Lem}
\label{lem:converse-of-Neeman's-localisation-theorem}
Assume that there is a diagram of triangle functors between compactly generated triangulated $k$-categories\emph{:}
\[
 \xymatrixcolsep{4pc}\xymatrix{\mathcal{T'} \ar[r]|{i_*=i_!} &\mathcal{T} \ar@/_1pc/[l]|{i^*} \ar@/^1pc/[l]|{i^!} \ar[r]|{j^*=j^!}  &\mathcal{T''} \ar@/^-1pc/[l]|{j_!} \ar@/^1pc/[l]|{j_*}
  }
\]
satisfying (R1) in Definition~\ref{defn:recollement}. If the first row restricts to a short exact sequence up to direct summands
\[
\xymatrix{
\mathcal{T'}^c & \mathcal{T}^c \ar[l]_{i^*} & \mathcal{T''}^c\ar[l]_{j_!},
}
\]
then the above diagram is a recollement.
\end{Lem}

\begin{proof}
Since $j_!|_{\mathcal{T}''^c}\colon \mathcal{T}''^c\to\mathcal{T}^c$ is fully faithful, it follows by \cite[Lemma 4.2 b)]{Keller94} that $j_!$ is fully faithful. Notice that the composition $i^*\circ j_!\colon \mathcal{T}'' \ra \mathcal{T}'$ commutes with infinite direct sums and sends compact objects to zero, so it is the zero functor. Therefore there exists a unique triangle functor $\overline{i^*}\colon\mathcal{T}/j_!(\mathcal{T}'') \ra \mathcal{T}'$ such that $\overline{i^*}\circ \pi=i^*$
\[
\xymatrix{
&\mathcal{T}\ar[rd]^{i^*} \ar[ld]_{\pi}&\\
\mathcal{T}/j_!(\mathcal{T}'')\ar[rr]^{\overline{i^*}} && \mathcal{T}'
}
\]
where $\pi$ is the quotient functor.
Since $\pi$ is dense and commutes with infinite direct sums, it follows that $\overline{i^*}$ also commutes with infinite direct sums. By our assumption and by Theorem~\ref{thm:Neeman's-localisation-theorem}, there is a commutative diagram
\[
\xymatrix{
& \mathcal{T}^c/j_!(\mathcal{T}''^c)\ar[ld]\ar[rd]&\\
(\mathcal{T}/j_!(\mathcal{T}''))^c \ar[rr]^{\overline{i^*}|} && \mathcal{T}'^c
}
\]
such that both the slanted functors are triangle equivalences up to direct summands. It follows that $\overline{i^*}|_{(\mathcal{T}/j_!(\mathcal{T}''))^c}\colon (\mathcal{T}/j_!(\mathcal{T}''))^c\to \mathcal{T}'^c$ is a triangle equivalence. By \cite[Lemma 4.2 c)]{Keller94}, $\overline{i^*}$ itself is a triangle equivalence. So the first row of the given diagram is a short exact sequence and the diagram is a recollement by Proposition~\ref{prop:recollement-vs-ses}.
\end{proof}

\medskip

\subsection{Derived categories of dg categories}\

DG $k$-categories can be considered as dg $k$-algebras with several objects. For a small dg $k$-category $\mathcal{A}$, we will denote by $\D\mathcal{A}$ the derived category of dg modules over $\mathcal{A}$ and by $\per\mathcal{A}$ the perfect derived category of $\mathcal{A}$. We will adopt the notation in Section~\ref{ss:derived-category-of-dg-algebras} and in \cite{Keller94,Keller06}.

Following \cite[Appendix B]{Drinfeld04}, we say that a dg $k$-category $\A$ is \emph{semi-free} if for any $X,Y\in\A$, the complex $\Hom_\A(X,Y)$ is the union of an increasing sequence of subcomplexes $M_i$, $i\in \mathbb{N}\cup \{0\}$, such that each quotient $M_{i+1}/M_i$ is free over $k$. By \cite[Lemma B.5]{Drinfeld04}, for any dg $k$-category $\A$, there is a semi-free dg $k$-category $\tilde{\A}$ together with a quasi-isomorphism $F\colon\tilde{\A}\to\A$, which we call a \emph{semi-free resolution} of $\A$. Let $X$ be the dg $\tilde{\A}$-$\A$-bimodule defined by $X(\tilde{A},A):=\A(A,F(\tilde{A}))$. Then the induction functor $\mathbf{L}T_X\colon \D(\tilde{\A})\to \D(\A)$ is a triangle equivalence with quasi-inverse the restriction functor $\mathbf{L}T_{Y}\colon\D(\A)\to\D(\tilde{\A})$, by \cite[Lemma 6.1 a) and Lemma 6.2 b)]{Keller94}, where $Y$ is the dg $\A$-$\tilde{\A}$-bimodule defined by $Y(A,\tilde{A})=\A(F(\tilde{A}),A)$.

Let $\mathcal{U}$ be a full subcategory of $\D\mathcal{A}$. A \emph{standard lift} of $\mathcal{U}$ is a pair $(\mathcal{B},X)$, where $\mathcal{B}$ is a full dg subcategory of $\C_{\dg}\mathcal{A}$ consisting of precisely one $\K$-projective resolution for each object of $\mathcal{U}$, and $X$ is the dg $\mathcal{B}$-$\mathcal{A}$-bimodule defined by
$X(A, B):=B(A)$ for $A\in\mathcal{A}$ and $B\in\mathcal{B}$. Then the triangle functor $\mathbf{L}T_{X}\colon \D\mathcal{B}\to\D\mathcal{A}$ restricts to an equivalence $\per\mathcal{B}\to\thick_{\D\mathcal{A}}(\mathcal{U})$. See \cite[Section 7]{Keller94}. Moreover, if $\A$ is semi-free, then we may assume that $X$ is $\K$-projective over $\B^{\oop}\ten_k \A$, in this case, $X$ is then $\K$-projective over $\B$ and $\mathbf{L}T_X\cong T_X$, by \cite[Lemma 6.2 a)]{Keller94}.

\begin{Prop}[{\cite[Theorem 1]{Yang12}}]
\label{prop4.2}
\label{prop:quotient==>recollement}
Let $\mathcal{A}$ be a dg $k$-category. Let $\mathcal{S}$ be a full subcategory of $\D\mathcal{A}$, such that $\per\mathcal{A} \subset \thick(\mathcal{S})$.
Let $(\mathcal{B},X)$ be a standard lift of $\mathcal{S}$. Then there is a dg $k$-category $\mathcal{C}$, a dg functor $F\colon\mathcal{B}\to\mathcal{C}$ and a recollement:
 \[
 \xymatrixcolsep{6pc}\xymatrix{\D\mathcal{C} \ar[r]|{i_*} &\D\mathcal{B} \ar@/_1pc/[l]_{i^*} \ar@/^1pc/[l]^{i^!} \ar[r]|{\mathbf{R}H_{X^T}\cong \mathbf{L}T_X}  &\D\mathcal{A}, \ar@/^-1pc/[l]_{\mathbf{L}T_{X^T}} \ar@/^1pc/[l]^{\mathbf{R}H_{X}}
} \]
where the triple $(i^*,i_*,i^!)$ is induced by the dg functor $F$.
\end{Prop}

\medskip

\subsection{Main result in terms of dg categories}\

In this subsection we show that Theorem~\ref{thm:s.e.s.-of-quotients-in-the-first-row} is equivalent to
\begin{Thm}\label{thm4.1}
\label{thm:s.e.s.-of-quotients-in-the-first-row-in-terms-of-derived-categories}
Let  $\A$, $\A'$ and $\A''$ be dg $k$-categories together with a recollement
\xymatrixcolsep{4pc}
\[ \xymatrix{\D(\A') \ar[r]|-{i_*} &\D(\A) \ar@/_1pc/[l]|-{i^*} \ar@/^1pc/[l]|-{i^!} \ar[r]|-{j^*}  &\D(\A''), \ar@/^-1pc/[l]|-{j_!} \ar@/^1pc/[l]|-{j_*}  }\]
where $i^*=T_V$ for a dg $\A$-$\A'$-bimodule $V$ which is $\K$-projective over $\A$, and $j_!=T_U$ for a dg $\A''$-$\A$-bimodule $U$ which is $\K$-projective over $\A''$.

Let $\cal S$, $\cal S'$ and $\cal S''$ be triangulated subcategories of $\D(\A)$, $\D(\A')$ and $\D(\A'')$, which contain $\per(\A)$, $\per(\A')$ and $\per(\A'')$, respectively, such that the first row of the above recollement restricts to a short exact sequence (respectively, a short exact sequence up to direct summands)
\begin{align}
\label{ses:algebraic-case-larger}
\xymatrix{
\cal S' & \cal S\ar[l]_{i^*} & \cal S''\ar[l]_{j_!}.
}
\end{align}
Then there is a short exact sequence (respectively, a short exact sequence up to direct summands):
\[ 
\xymatrix{
\cal S'/\per(\A') & \cal S/\per(\A) \ar[l]_{\bar{i}^*} & \cal S''/\per (\A'')\ar[l]_{\bar{j}_!}.
} 
\]
\end{Thm}

It is clear that Theorem~\ref{thm:s.e.s.-of-quotients-in-the-first-row} implies Theorem~\ref{thm:s.e.s.-of-quotients-in-the-first-row-in-terms-of-derived-categories}. Now we prove that Theorem~\ref{thm:s.e.s.-of-quotients-in-the-first-row-in-terms-of-derived-categories} implies Theorem~\ref{thm:s.e.s.-of-quotients-in-the-first-row}. 
Keep the notation and assumptions in Theorem~\ref{thm:s.e.s.-of-quotients-in-the-first-row}. 

First, by \cite[Theorem 4.3]{Keller94}, there is a dg $k$-category $\A$ such that $\mathcal{T}\simeq \D\A$. Replacing $\A$ with a semi-free resolution of $\A$, we may assume that $\A$ is semi-free and that $\mathcal{T}=\D\A$. 

Secondly, let $\mathcal{W}$ be a set of compact generators of $\mathcal{T}''$.  Consider the standard lift $(\mathcal{A}'',U)$ of $\mathcal{U}=j_!(\mathcal{W})$. Since $\mathcal{U} \subset \per(\A)$, by the property of lift, there is a triangle equivalence $\mathbf{L}T_{U}\colon\D\mathcal{\A''}\rightarrow \Loc(\mathcal{U})$. Since $j_!$ is fully faithful, it induces a triangle equivalence $j_!\colon\mathcal{T}''\rightarrow \Loc(\mathcal{U})$.
Therefore, we can replace $\mathcal{T}''$ by $\D\A''$, and we may assume $j_!=\mathbf{L}T_{U}$. Then by replacing $U$ with a $\K$-projective resolution over $\A''^{\oop}\ten_k \A$, we may assume that $U$ is $\K$-projective over $\A''^{\oop}\ten_k \A$, and hence it is $\K$-projective over $\A''$ and $j_!=T_U$. 

Thirdly, according to \cite[Theorem 4]{NicolasSaorin09},
there is a recollement which is equivalent to the recollement \eqref{recol:abstract-cat}
\xymatrixcolsep{7pc}
\begin{eqnarray*} 
\label{recol:dg-cat}
\xymatrix{\D(\A'_1) \ar[r] &\D(\A) \ar@/_1pc/[l]|-{\mathbf{L}T_{V_1}} \ar@/^1pc/[l] \ar[r]  &\D(\A''). \ar@/^-1pc/[l]|{T_U} \ar@/^1pc/[l]  }
\end{eqnarray*}
\vspace{3pt}

\noindent Here the dg $\A$-$\A'_1$-bimodule $V_1$ is induced by a homological eipmorphism $F\colon\A \rightarrow \A'_1$
which is bijective on objects, namely, $V_1(A,A'_1):=\mathcal{A}'_1(A'_1,FA)$. Further, let $G\colon \A'\to\A'_1$ be a semi-free resolution of $\A'$ and let $V_2$ be the dg $\A'_1$-$\A'$-bimodule defined by  $V_2(A'_1,A')=\A'_1(G(A'),A'_1)$. Then $\mathbf{L}T_{V_2}$ is a triangle equivalence, and by \cite[Lemma 6.3 b)]{Keller94}, there is a dg $\A$-$\A'$-bimodule $V$ such that $\mathbf{L}T_V\cong \mathbf{L}T_{V_2}\circ\mathbf{L}T_{V_1}$. Finally, replacing $V$ with its $\K$-projective resolution over $\A^{\oop}\ten_k \A'$, we have that $V$ is $\K$-projective over $\A$ and $\mathbf{L}T_V\cong T_V$, and thus we obtain a recollement satisfying the conditions in the beginning of Theorem~\ref{thm:s.e.s.-of-quotients-in-the-first-row-in-terms-of-derived-categories}.

\subsection{Proof of Theorem~\ref{thm:s.e.s.-of-quotients-in-the-first-row-in-terms-of-derived-categories}}\

We first remark that it is enough to prove the case when \eqref{ses:algebraic-case-larger} is a short exact sequence up to direct summands.

The following observation will be useful.
\begin{Lem}\label{up}
Assume we have the following commutative diagram of triangulated categories and triangle functors:

\[  \xymatrixcolsep{4pc}\xymatrix{\mathcal{T'} & \mathcal{T} \ar[l]_{F} & \mathcal{T''} \ar[l]_{G}
 \\ \mathcal{P'} \ar[u]^{\beta'} & \mathcal{P} \ar[l]_{H} \ar[u]^{\beta} & \mathcal{P''} \ar[l]_{K} \ar[u]^{\beta''}
}  \]
such that the functors $\beta$, $\beta'$ and $\beta''$ are fully faithful, moreover they are dense up to taking direct summands. Then the upper row is a short exact sequence up to taking direct summands if and only if the bottom row is a short exact sequence up to taking direct summands.
\end{Lem}

\begin{proof}
Assume the upper row is a short exact sequence up to taking direct summands. It is obvious that $K$ is fully faithful and $H\circ K=0$. By our assumption on $\beta''$, the natural functor $\mathcal{T}/\mathcal{P''}\ra\mathcal{T}/\mathcal{T''}$ is an equivalence. Since $\overline{F}: \mathcal{T}/\mathcal{T''}\ra\mathcal{T'}$ is fully faithful, we have $\overline{H}: \mathcal{P}/\mathcal{P''}\ra\mathcal{P'}$ is also fully faithful and is dense up to taking direct summands.

On the other hand, assume the bottom row is a short exact sequence up to taking direct summands. It is easy to show $G$ is fully faithful and $F\circ G=0$. The equivalence $\mathcal{T}/\mathcal{P''}\ra\mathcal{T}/\mathcal{T''}$ and the fact that $\beta$ is an equivalence up to taking direct summands imply $\overline{\beta}: \cal{P}/\cal{P''}\ra\cal{T}/\cal{T''}$ is an equivalence up to taking direct summands. Since $\beta'$ and $\overline{F}: \cal{T}/\cal{T''}\ra\cal{T}$ are both equivalences up to taking direct summands, we are done.
\end{proof}

To prove Theorem~\ref{thm:s.e.s.-of-quotients-in-the-first-row-in-terms-of-derived-categories}, we will identify $\mathcal{S'}/\per \mathcal{A'}$, $~\mathcal{S}/\per \mathcal{A}$~and~$\mathcal{S''}/\per\mathcal{A''}$ with the subcategories of compact objects of the derived categories of certain dg categories. Then we will construct a recollement of these derived categories. First, we can choose nice lifts of $\mathcal{S'}$, $\mathcal{S}$ and $\mathcal{S''}$.

\begin{Lem}\label{lemma4.2}
There are standard lifts of $\mathcal{S'}$,~$\mathcal{S}$ and $\mathcal{S''}$, denoted by $(\mathcal{B'},X')$, $(\mathcal{B},X)$ and $(\mathcal{B''},X'')$, respectively, such that    $T_U(\mathcal{B''})\subset \mathcal{B}$ and $T_V(\mathcal{B})\subset \mathcal{B'}$~.
\end{Lem}

\begin{proof}
Let $(\mathcal{B''},X'')$ be a standard lift of $\mathcal{S''}$.
For $S''\in \mathcal{S''}$, let $\mathbf{p}S''\rightarrow S''$ be a $\mathcal{K}$-projective resolution. Since $T_U$ preserves acyclicity and $\K$-projectivity (see \cite[Lemma 6.1]{Keller94}), it follows that $T_U(\mathbf{p}S'')\rightarrow T_U(S'')$ is a $\K$-projective resolution of $T_U(S'')\in \mathcal{S}$. We can extend $\{T_U(\mathbf{p}S'')\mid S''\in\mathcal{S}''$ to a standard lift of $\mathcal{B}$. Then clearly $T_U(\mathcal{B''})\subset \mathcal{B}$. Similarly, one can find a standard lift $(\mathcal{B}',X')$ of $\mathcal{S}'$ such that $T_V(\mathcal{B})\subset \mathcal{B'}$.
\end{proof}

Now we have two dg functors: $T_U:\mathcal{B''}\rightarrow\mathcal{B}$ and $T_V:\mathcal{B}\rightarrow\mathcal{B'}$. They induce a $\mathcal{B''}^{op}\otimes \mathcal{B}$-module $\widetilde{\mathcal{U}}$ and a $\mathcal{B}^{op}\otimes \mathcal{B'}$-module $\widetilde{V}$
which are defined by $\widetilde{U}(B'', B)=\mathcal{B}(B, T_U(B''))$ and $\widetilde{V}(B,~B')=\mathcal{B'}(B',~T_V(B))$ respectively.

For any $\mathbf{p}S''\in\mathcal{B''} $, we have $T_U\circ \mathbf{L}T_{X''}((\mathbf{p}S'')^{\wedge})=T_U(\mathbf{p}S'') $ and $\mathbf{L}T_X\circ \mathbf{L}T_{\widetilde{U}}((\mathbf{p}S'')^{\wedge})=\mathbf{L}T_X((T_U(\mathbf{p}S''))^{\wedge})=T_U(\mathbf{p}S'')$. Thus $T_U\circ \mathbf{L}T_{X''}=\mathbf{L}T_X\circ \mathbf{L}T_{\widetilde{U}} $. Similarly, $T_V\circ \mathbf{L}T_X=\mathbf{L}T_{X'}\circ \mathbf{L}T_{\widetilde{V}}$. Then we have the following commutative diagram:

\begin{equation}
\label{equ:commutative_diagram_lift}
\xymatrixcolsep{4pc}\xymatrix{\D\mathcal{A'} & \D\mathcal{A} \ar[l]_{T_V} & \D\mathcal{A''} \ar[l]_{T_U}
 \\ \D\mathcal{B'} \ar[u]^{\mathbf{L}T_{X'}} & \D\mathcal{B} \ar[l]_{\mathbf{L}T_{\widetilde{V}}} \ar[u]^{\mathbf{L}T_X} & \D\mathcal{B''}. \ar[l]_{\mathbf{L}T_{\widetilde{U}}} \ar[u]^{\mathbf{L}T_{X''}}
}  \end{equation}

\

We will show that the bottom row of the diagram above can be extended to a recollement. Notice that $\mathbf{L}T_{\widetilde{U}}(\per \ \mathcal{B}'')\subset \per \ \mathcal{B}$ and $\mathbf{L}T_{\widetilde{V}}(\per \ \mathcal{B}^{})\subset \per \ \mathcal{B}'$, then by \cite[Lemma 6.4]{Keller94}, $\mathbf{R}H_{\widetilde{U}}=\mathbf{L}T_{\widetilde{U}^T}$,~$\mathbf{R}H_{\widetilde{V}}=\mathbf{L}T_{\widetilde{V}^T}$, thus there is a diagram of triangle functors:

\begin{align}
\label{recol:Kosuzl-dual}
\xymatrixcolsep{6pc}\xymatrix{\D\mathcal{B'} \ar[r]^{\mathbf{R}H_{\widetilde{V}}=\mathbf{L}T_{\widetilde{V}^T}} &\D\mathcal{B} \ar@/_2pc/[l]_{\mathbf{L}T_{\widetilde{V}}} \ar@/^2pc/[l]^{\mathbf{R}H_{\widetilde{V}^T}}
\ar[r]^{\mathbf{R}H_{\widetilde{U}}=\mathbf{L}T_{\widetilde{U}^T}}  &\D\mathcal{B''} \ar@/^-2pc/[l]_{\mathbf{L}T_{\widetilde{U}}}
\ar@/^2pc/[l]^{\mathbf{R}H_{\widetilde{U}^T}}
  }
\end{align}

\begin{Prop}
\label{prop:recollement-of-kosuzl-dual}
The diagram \eqref{recol:Kosuzl-dual} is a recollement.
\end{Prop}
\begin{proof}
By the construction of $\mathcal{B}',\mathcal{B},\mathcal{B}''$, we have the following commutative diagram
\begin{align}
\label{diag:perfectisation-of-S}
\xymatrixcolsep{4pc}\xymatrix{
\mathcal{S}' \ar@{^{(}->}[d]& \mathcal{S}\ar[l]_{T_V}\ar@{^{(}->}[d] & \mathcal{S}''\ar[l]_{T_U}\ar@{^{(}->}[d]\\
\thick(\mathcal{S'}) & \thick(\mathcal{S}) \ar[l]_{T_V} & \thick(\mathcal{S''}) \ar[l]_{T_U}
 \\ \per\mathcal{B'} \ar[u]^{\mathbf{L}T_{X'}}_{\simeq} & \per\mathcal{B} \ar[l]_{\mathbf{L}T_{\widetilde{V}}} \ar[u]^{\mathbf{L}T_X}_{\simeq} & \per\mathcal{B''} \ar[l]_{\mathbf{L}T_{\widetilde{U}}} \ar[u]^{\mathbf{L}T_{X''}}_{\simeq}
}
\end{align}
Since the first row is a short exact sequence up to direct summand, so is the third row by Lemma \ref{up}. Then the desired result follows immediately from Lemma~\ref{lem:converse-of-Neeman's-localisation-theorem}.
\end{proof}

By Propositions~\ref{prop:quotient==>recollement}~and~\ref{prop:recollement-of-kosuzl-dual}, we obtain a diagram whose rows and columns are recollemnets:
\begin{align}
\label{diag:big-diagram}
\xymatrixrowsep{5pc} \xymatrixcolsep{8pc}\xymatrix{
\D\mathcal{A'} \ar[r]^{\mathbf{R}H_{V} \simeq \mathbf{L}T_{V^T}}  \ar@/^-1pc/[d]_{\mathbf{L}T_{X'^{T}}} \ar@/^1pc/[d]^{\mathbf{R}H_{X'}}
&\D\mathcal{A} \ar@/^-2pc/[l]_{\mathbf{L}T_{V}} \ar@/^2pc/[l]_{\mathbf{R}H_{V^T}} \ar[r]^{\mathbf{R}H_{U} \simeq \mathbf{L}T_{U^T}} \ar@/^-1pc/[d]_{\mathbf{L}T_{X^{T}}} \ar@/^1pc/[d]^{\mathbf{R}H_{X}}
&\D\mathcal{A''} \ar@/^-2pc/[l]_{\mathbf{L}T_{U}} \ar@/^2pc/[l]_{\mathbf{R}H_{U^T}} \ar@/^-1pc/[d]_{\mathbf{L}T_{X''^{T}}} \ar@/^1pc/[d]^{\mathbf{R}H_{X''}}  \\
\D\mathcal{B'} \ar[u] \ar[r]^{\mathbf{R}H_{\widetilde{V}} \simeq \mathbf{L}T_{\widetilde{V}^T}} \ar@/^-1pc/[d]_{\LT_{Y'}} \ar@/^1pc/[d]
&\D\mathcal{B} \ar@/^-2pc/[l]_{\mathbf{L}T_{\widetilde{V}}} \ar@/^2pc/[l]_{\mathbf{R}H_{\widetilde{V}^T}} \ar[u] \ar[r]^{\mathbf{R}H_{\widetilde{U}} \simeq \mathbf{L}T_{\widetilde{U}^T}} \ar@/^-1pc/[d]_{\LT_{Y}} \ar@/^1pc/[d]
&\D\mathcal{B''} \ar[u] \ar@/^-2pc/[l]_{\mathbf{L}T_{\widetilde{U}}} \ar@/^2pc/[l]_{\mathbf{R}H_{\widetilde{U}^T}} \ar@/^-1pc/[d]_{\LT_{Y''}} \ar@/^1pc/[d] \\
\D\mathcal{C'} \ar[u]
&\D\mathcal{C}  \ar[u]
&\D\mathcal{C''} \ar[u]
}
\end{align}
Here $\LT_{Y}$ is induced by the dg functor $F\colon\mathcal{B}\rightarrow\mathcal{C}$. $\LT_{Y'}$ and $\LT_{Y''}$ are induced in similar ways.
The following proposition implies our main theorem.

\begin{Prop}\label{prop4.3}
There is a recollement
\[
 \xymatrixrowsep{4pc} \xymatrixcolsep{5pc}\xymatrix{
\D\mathcal{C'}  \ar[r]^{\sigma'}
&\D\mathcal{C} \ar@/^-1pc/[l]_{\sigma} \ar@/^1pc/[l] \ar[r]^{\tau'}
&\D\mathcal{C''} \ar@/^-1pc/[l]_{\tau} \ar@/^1pc/[l]
}  \]
such that the following diagram
\[
\xymatrixcolsep{4pc}\xymatrix{
  \D\mathcal{B'}  \ar@/^-1pc/[d]_<<<{\mathbf{L}T_{Y'}}
    &\D\mathcal{B} \ar@/^-1pc/[l]_{\mathbf{L}T_{\w{V}}}  \ar@/^-1pc/[d]_<<<{\mathbf{L}T_{Y}}
    &\D\mathcal{B''} \ar@/^-1pc/[l]_{\mathbf{L}T_{\w{U}}}  \ar@/^-1pc/[d]_<<<{\mathbf{L}T_{Y''}}   \\
 \D\mathcal{C'}
   &\D\mathcal{C} \ar@/^-1pc/[l]_{\sigma}
   &\D\mathcal{C''} \ar@/^-1pc/[l]_{\tau}
}
\]
commutates.
\end{Prop}

\

\begin{proof}

We first show that the following diagram is commutative:
\[
\xymatrixcolsep{4pc}\xymatrix{
  \D\mathcal{A'}  \ar@/^-1pc/[d]_<<<{\mathbf{L}T_{X'^{T}}}
    &\D\mathcal{A} \ar@/^-1pc/[l]_{\mathbf{L}T_{V}}  \ar@/^-1pc/[d]_<<<{\mathbf{L}T_{X^{T}}}
    &\D\mathcal{A''} \ar@/^-1pc/[l]_{\mathbf{L}T_{U}}  \ar@/^-1pc/[d]_<<<{\mathbf{L}T_{X''^{T}}}   \\
 \D\mathcal{B'}
   &\D\mathcal{B} \ar@/^-1pc/[l]^{\mathbf{L}T_{\widetilde{V}}}
   &\D\mathcal{B''}. \ar@/^-1pc/[l]^{\mathbf{L}T_{\widetilde{U}}}
}
\]
By the definition of lift, $\mathbf{L}T_X$ induces an equivalence between $\per\mathcal{B}$ and $\thick(\mathcal{S})$ and the quasi-inverse of the equivalence is given by $\mathbf{R}H_X$ (see  \cite[Section 7.3]{Keller94}).
For any $A''\in \A''$, we have $\L T_{X^{T}}\circ \L T_{U}(A''^{\wedge})=\bf{R} H_{X}(T_U(A''^{\wedge}))=(T_U(A''^{\wedge}))^{\wedge}$, because $\L T_{X^{T}}|_{\per \A''}=\bf{R}H_{X}|_{\per\A''}$ by \cite[Lemma 6.4]{Keller94}. Similarly, $\L T_{\w{U}} \circ \L T_{X''^{T}}(A''^{\wedge})=\L T_{\w{U}} \circ \bf{R} H_{X''}(A''^{\wedge})=\mathbf{L}T_{\widetilde{U}}((A''^{\wedge})^{\wedge}) =(T_U(A''^{\wedge}))^{\wedge}$. Then $\mathbf{L}T_{X^T}\circ \mathbf{L}T_{U}=\mathbf{L}T_{\widetilde{U}} \circ \mathbf{L}T_{X''^{T}}$. So the right part of the diagram above is commutative, and the commutativity of left part can be shown similarly.
By adjunction, we have a new commutative diagram:

\[
\xymatrixcolsep{4pc}\xymatrix{
\D\mathcal{A'} \ar[r]^{\mathbf{R}H_{V} }
&\D\mathcal{A}  \ar[r]^{\mathbf{R}H_{U} }
&\D\mathcal{A''}   \\
\D\mathcal{B'} \ar[u]^{\mathbf{L}T_{X'}} \ar[r]^{\mathbf{R}H_{\widetilde{V}} }
&\D\mathcal{B}  \ar[u]^{\mathbf{L}T_{X}} \ar[r]^{ \mathbf{R}H_{\widetilde{U}} }
&\D\mathcal{B''}. \ar[u]^{\mathbf{L}T_{X''}}
}
\]
Combining this diagram and \eqref{equ:commutative_diagram_lift}, we obtain the following left recollement
\begin{align}
\label{leftrecollementker}
\xymatrixcolsep{4pc}\xymatrix{\ke(\L T_{X'}) \ar[r]|{\bf{R} H_{\w{V}}} &\ke(\L T_{X}) \ar@/_1.5pc/[l]_{\L T_{\w{V}}}  \ar[r]|{\bf{R} H_{\w{U}}}  &\ke(\L T_{X''}). \ar@/^-1.5pc/[l]_{\L T_{\w{U}}}
  }
\end{align}
Note that $\ke(\LT_{X})$, $\ke(\LT_{X'})$ and $\ke(\LT_{X''})$ are triangle equivalent to $\D\cal{C}$,   $\D\cal{C'}$ and $\D\cal{C''}$, respectively. Since $\bf{R} H_{\w{V}}\cong \L T_{\w{V}^{T}} $ and  $\bf{R} H_{\w{U}}\cong \L T_{\w{U}^{T}}$ commutate with infinite direct sums, $\bf{R} H_{\w{V}}|_{\ke(\L T_{X'})}$ and $\bf{R} H_{\w{U}}|_{\ke(\L T_{X})}$
admit right adjoint by Lemma \ref{adjoint}.
Thus \eqref{leftrecollementker} extends to a recollement by Proposition \ref{prop:recollement-vs-ses}.
It is equivalent to a recollement of $\D\cal{C}$ in terms of $\D\cal{C'} $ and $\D\cal C''$, as deserved.
The second assertion obtained by the following commutative diagram
\[
\xymatrixcolsep{4pc}\xymatrix{
\D\mathcal{B'} \ar[r]^{\mathbf{R}H_{\w{V}} }
&\D\mathcal{B}  \ar[r]^{\mathbf{R}H_{\w{U}} }
&\D\mathcal{B''}   \\
\D\mathcal{C'} \ar[u]^{\mathbf{L}T_{Y'^{T}}} \ar[r]^{\sigma'}
&\D\mathcal{C}  \ar[u]^{\mathbf{L}T_{Y^{T}}} \ar[r]^{ \tau'}
&\D\mathcal{C''}. \ar[u]^{\mathbf{L}T_{Y''^{T}}}
}
\]
\end{proof}

\

Now we are ready to prove Theorem~\ref{thm:s.e.s.-of-quotients-in-the-first-row-in-terms-of-derived-categories}.
\begin{proof}[Proof of Theorem~\ref{thm:s.e.s.-of-quotients-in-the-first-row-in-terms-of-derived-categories}]

By Proposition \ref{prop4.3}, we have the following commutative diagram
\[\xymatrixrowsep{3pc} \xymatrixcolsep{5pc}\xymatrix{
 \per \mathcal{A'}  \ar@/^-1pc/[d]_<<<<<{\mathbf{L}T_{X^{'T}}}
    &\per\mathcal{A} \ar@/^-1pc/[l]_{\mathbf{L}T_{V}}  \ar@/^-1pc/[d]_<<<<<{\mathbf{L}T_{X^{T}}}
    &\per\mathcal{A''} \ar@/^-1pc/[l]_{\mathbf{L}T_{U}}  \ar@/^-1pc/[d]_<<<<<{\mathbf{L}T_{X^{''T}}}   \\
 \per\mathcal{B'}   \ar@/^-1pc/[d]_<<<<<{\mathbf{L}T_{Y'}}
   &\per\mathcal{B} \ar@/^-1pc/[l]_{\mathbf{L}T_{\widetilde{V}}}  \ar@/^-1pc/[d]_<<<<<{\mathbf{L}T_{Y}}
   &\per\mathcal{B''} \ar@/^-1pc/[l]_{\mathbf{L}T_{\widetilde{U}}} \ar@/^-1pc/[d]_<<<<<{\mathbf{L}T_{Y''}}
  \\
 \per\mathcal{C'}
 & \per\mathcal{C} \ar@/^-1pc/[l]_{\sigma}
 & \per\mathcal{C''} \ar@/^-1pc/[l]_{\tau}
}  \]
This diagram induces a commutative diagram

\[\xymatrixrowsep{3pc} \xymatrixcolsep{4pc}\xymatrix{
 \per\mathcal{B'}/\mathbf{L}T_{X^{'T}}(\per\mathcal{A'})   \ar@/^-1pc/[d]_<<<<<{\overline{\mathbf{L}T_{Y'}}}
   &\per\mathcal{B}/\mathbf{L}T_{X^{T}}(\per\mathcal{A}) \ar@/^-1pc/[l]_{\overline{\mathbf{L}T_{\widetilde{V}}}}  \ar@/^-1pc/[d]_<<<<<{\overline{\mathbf{L}T_{Y}}}
   &\per\mathcal{B''}/\mathbf{L}T_{X^{''T}}(\per\mathcal{A''}) \ar@/^-1pc/[l]_{\overline{\mathbf{L}T_{\widetilde{U}}}} \ar@/^-1pc/[d]_<<<<<{\overline{\mathbf{L}T_{Y''}}}
  \\
 \per\mathcal{C'}
 & \per\mathcal{C} \ar@/^-1pc/[l]_{\sigma}
 & \per\mathcal{C''}, \ar@/^-1pc/[l]_{\tau}
}  \]
where the induced functor $\overline{\LT_{Y}}$, $\overline{\LT_{Y'}}$ and $\overline{\LT_{Y''}}$ are fully faithful and are dense up to direct summands. The bottom row is an exact sequence up to direct summands, so is the upper row by Lemma \ref{up}.
The commutative diagram~\eqref{diag:perfectisation-of-S} induces a commutative diagram
\[
 \xymatrix@C=4pc{\per\mathcal{B'}/\mathbf{L}T_{X^{'T}}(\per\mathcal{A'}) & \per\mathcal{B}/\mathbf{L}T_{X^{T}}(\per\mathcal{A}) \ar[l]_{\overline{\LT_{\w{V}}}} &\per\mathcal{B''}/\mathbf{L}T_{X^{''T}}(\per\mathcal{A''}) \ar[l]_{\overline{\LT_{\w{U}}}} \\
\mathcal{S'}/\per\mathcal{A'} \ar[u]_{\mathbf{L}T_{X^{'T}}} & \mathcal{S}/\per\mathcal{A} \ar[l]_{i^*} \ar[u]_{\mathbf{L}T_{X^{T}}} & \mathcal{S''}/\per\mathcal{A''} \ar[l]_{j_!} \ar[u]_{\mathbf{L}T_{X^{''T}}}
}
\]
where the vertical functors are all triangle equivalences up to direct summands.
By Lemma \ref{up}, the bottom row is an exact sequence up to direct summands.
\end{proof}

\bigskip

\section{An application  to algebraic geometry}\label{Sect: geometrical application}

In this section, we present an application to algebraic geometry.

Let $X$ be a scheme. Denote by $O_X$-Mod (resp. $\mathrm{Qcoh}(X)$    ) the category of $O_X$-modules (resp. quasi-coherent sheaves) on $X$. When $X$ is noetherian, $\mathrm{Coh}(X)$ denotes the category of coherent sheaves.   We will denote $\D(X)=\D_{\mathrm{Qcoh}}(O_X$-Mod) to be    the unbounded derived category of  $O_X$-Mod with quasi-coherent cohomology, which is equivalent to the unbounded derived category $\D(\mathrm{Qcoh}(X))$  of quasi-coherent sheaves only when $X$ is quasi-compact and separated \cite{ATJLL97, BN93, Keller01}.

Recall that for a noetherian scheme $X$, a bounded complex of coherent sheaves will be called a perfect complex if it is locally quasi-
isomorphic to a bounded complex of locally free sheaves of finite type. Denote by $\per(X)$ the full subcategory of $\D^b(\mathrm{Coh}(X))$ consisting of perfect complexes.

 Following \cite{Orlov04}, we will say that a scheme $X$ over a field $k$
satisfies the condition (ELF) if it is
  separated, noetherian, of finite Krull dimension, and the category of coherent sheaves
$\mathrm{Coh}(X)$ has enough locally free sheaves.
\begin{Def}[\cite{Orlov04}] Let $X$ be a scheme satisfying (ELF).  The bounded singularity category of $X$ is  defined to be
$$\D_{\sg}(X)=\D^b(\mathrm{Coh}(X))/\per(X).$$

\end{Def}
Let  $U\stackrel{i}{\hookrightarrow} X$ an open subscheme and write  $Z=X\backslash
U$. Let $$\D_Z^b(\mathrm{Coh}(X))=\{C\in \D^b(\mathrm{Coh}(X))\ |\ H^i(C) \ \mathrm{has \ its \ support \ in}\ Z, \forall i\in \mathbb{Z} \},$$
and $\per_Z(X)$ be the full triangulated category of $\per(X)$ consisting of perfect complexes supported in $Z$.
Then the bounded singularity category of $X$ supported in $Z$ is by definition the Verdier quotient
$$\D_{\sg, Z} (X):=\D_Z^b(\mathrm{Coh}(X))/\per_Z(X).$$
Let $i^*: \D_{\sg}(X)\to \D_{\sg}(U)$ be induced by the pullback functor $\mathrm{Coh}(X)\to \mathrm{Coh}(U)$ and $j_!: \D_{\sg, Z}^b(X)\to \D_{\sg}(X)$ be  induced by the inclusion functor $\D_Z^b(\mathrm{Coh}(X))\hookrightarrow \D^b(\mathrm{Coh}(X))$.

The following result relates these three singularity categories  $\D_{\sg}(X), \D_{\sg}(U)$ and $\D_{\sg, Z}^b(X)$ generalising
\cite[Proposition 2.7]{Orlov04} and \cite[Proposition 2.7]{Orlov10}.
\begin{Thm}[{\cite[Theorem 1.3]{Chen10}}]\label{Thm: Ses of Chen}
      There exists a short exact sequence   of    triangulated categories:
\[ \xymatrix{\D_{\sg}(U) & \D_{\sg}(X) \ar[l]|{i^*} & \D_{\sg, Z} (X) \ar[l]|{j_!}.  } 
\]

  \end{Thm}
The rest of this section is to give a proof of the above result using   Theorem~\ref{thm:s.e.s.-of-quotients-in-the-first-row}. 

\medskip

Let $X$ be a quasi-compact quasi-separated scheme with a quasi-compact open subscheme $i: U\hookrightarrow X$ with the closed complement $Z=X\backslash U$.
    A new  recollement is discovered in  \cite{Jorgensen09}, which has
   the form
\begin{equation}\label{equ:Jorgensen}  \xymatrix{\D(U) \ar[r]|{i_*} &\D(X) \ar@/_1pc/[l]|{i^*} \ar@/^1pc/[l]|{i^!} \ar[r]|{j^*}
&\D_Z(X) \ar@/^-1pc/[l]|{j_!} \ar@/^1pc/[l]|{j_*},
  }   \end{equation}
where  $i^*$ is the pull-back functor, $i_*$ is the push-forward functor, $j_!$ is the inclusion functor, and $\D_Z(X)=\{C\in \D(X)\ |\ H^i(C) \ \mathrm{has \ its \ support \ in}\ Z, \forall i\in \mathbb{Z} \}$.

Recall that when $X$ is quasi-compact and quasi-separated, $\D(X)$ is compactly generated whose full subcategory of compact objects is exactly $\per(X)$ \cite[Theorem 3.1.1]{BVB03}.
Thus the above recollement will induce a short exact sequence  of triangulated categories:
 \[ \xymatrix{\per(U) & \per(X) \ar[l]|{i^*} & \per_Z(X) \ar[l]|{j_!} },
\]
where the full subcategory of compact objects in $\D_Z(X)$ is identified with $\per_Z(X)$ by \cite[Theorem 6.8]{Rouquier08}.

Assume now our scheme $X$ satisfies the condition (ELF). By \cite[Lemma 2.2]{Orlov10},  there exists  a short exact sequence   of triangulated categories:
\[ \xymatrix{\D^b(\mathrm{Coh}(U)) & \D^b(\mathrm{Coh}(X)) \ar[l]|{i^*} &  \D^b_Z(\mathrm{Coh}(X)) \ar[l]|{j_!}. } 
\]

Now it is ready to see that the hypothesis of Theorem~\ref{thm:s.e.s.-of-quotients-in-the-first-row} holds and as a consequence, we get   a short exact sequence, up to direct summands,  of   of triangulated categories:
\[ \xymatrix{\D_{\sg}(U) & \D_{\sg}(X) \ar[l]|{i^*} & \D_{\sg, Z}(X) \ar[l]|{j_!}. } 
\]
 Since $i^*: \D^b(\mathrm{Coh}(X))\to \D^b(\mathrm{Coh}(U))$ is dense, so is    $i^*: \D_{\sg}(X)\to \D_{\sg}(U)$  and the above sequence is a short exact sequence.

We are done.

\smallskip

We conclude this section by a remark.
\begin{Rem}
\begin{itemize}
 \item[(a)] It is obvious that the essential image of $ j_!: \D_{\sg, Z}(X) \to \D_{\sg}(X)  $ is generated by $\mathrm{Coh}_Z(X)$, where $\mathrm{Coh}_Z(X)$ is the category of coherent sheaves supported in $Z$.  So we have equivalences:
     $$\D_{\sg}(X)/\mathrm{thick}(\mathrm{Coh}_Z(X))\cong \D_{\sg}(X)/ (\mathrm{Coh}_Z(X)) \cong \D_{\sg}(U),$$
     which is the original form of \cite[Theorem 1.3]{Chen10}.
 
 \item[(b)] 
 Our proof of Theorem~\ref{Thm: Ses of Chen} follows the same  spirit of \cite[Proposition 6.9]{Krause05}. While \cite{Krause05} uses explicit models for ``large" singularity categories, our proof uses dg lifts.  We  mention that the diagram \eqref{diag:big-diagram} appearing in the proof of Theorem~\ref{thm:s.e.s.-of-quotients-in-the-first-row} can be identified with the diagram in    \cite[Page 1149]{Krause05} if we use appropriate explicit  models.

\end{itemize}
\end{Rem}

\bigskip



\end{document}